\documentclass{article}
\usepackage[english]{babel} 		
\usepackage{fullpage}				
\usepackage[utf8]{inputenc}			
\usepackage{amsmath, amssymb,amsthm}
\usepackage{graphicx}				
\usepackage{hyperref}				
\usepackage{xcolor}
\usepackage{multirow}
\usepackage{tikz}
\usepackage{ytableau}
\usepackage{mathrsfs}
\usepackage{thmtools}
\usepackage{thm-restate}
\usepackage{varwidth}
\usepackage{comment}
\usepackage{mathtools}
\usepackage[mathscr]{euscript}
\usepackage{bbm}
\usepackage{multicol}
\usepackage{tikz-cd}
\usepackage{ytableau}
\usepackage{enumerate}
\theoremstyle{definition}
\newtheorem{thm}{Theorem}[section]
\newtheorem{lem}[thm]{Lemma}
\newtheorem*{lem*}{Lemma}
\newtheorem*{thm*}{Theorem}
\newtheorem{prop}[thm]{Proposition}

\newtheorem{cor}[thm]{Corollary}

\newtheorem{defn}[thm]{Definition}
\newtheorem*{remark*}{Remark}
\newtheorem{remark}{Remark}
\newtheorem{example}{Example}

\newtheorem{cor/defn}[thm]{Corollary/Definition}

\DeclareMathOperator{\Par}{\mathbb{Y}}

\DeclareMathOperator{\SSYT}{\mathrm{SSYT}}

\DeclareMathOperator{\sH}{\mathscr{H} }

\DeclareMathOperator{\sP}{\mathscr{P} }

\DeclareMathOperator{\Dim}{\mathrm{dim}}
\DeclareMathOperator{\Hom}{\mathrm{Hom}}

\DeclareMathOperator{\Char}{\mathrm{char}}
\DeclareMathOperator{\GL}{\mathrm{GL}}
\DeclareMathOperator{\bor}{\mathrm{B}}
\DeclareMathOperator{\para}{\mathrm{P}}
\DeclareMathOperator{\tor}{\mathrm{H}}
\DeclareMathOperator{\levi}{\mathrm{L}}

\DeclareMathOperator{\sort}{sort}
\DeclareMathOperator{\rev}{rev}

\DeclareMathOperator{\Inv}{\mathrm{Inv}}
\DeclareMathOperator{\inv}{inv}

\DeclareMathOperator{\End}{\mathrm{End}}
\newcommand{\y}{\mathscr{Y}}

\DeclareMathOperator{\mathleft}{left}
\DeclareMathOperator{\mathright}{right}
\DeclareMathOperator{\arm}{arm}
\DeclareMathOperator{\leg}{leg}
\DeclareMathOperator{\Des}{Des}
\DeclareMathOperator{\maj}{maj}
\DeclareMathOperator{\coinv}{coinv}
\DeclareMathOperator{\mathleg}{lg}
\DeclareMathOperator{\Comp}{\mathrm{Comp}}
\DeclareMathOperator{\Compred}{\mathrm{Comp}^{\textit{red}}}

\DeclareMathOperator{\Exp}{\mathrm{Exp}}

\usepackage[
    backend=bibtex,
    maxnames=5,
    style=alphabetic
  ]{biblatex}
\usepackage{graphicx} 

\title{Almost Symmetric Schur Functions}
\author{Milo Bechtloff Weising}
\date{\today}

\begin{document}

\maketitle

\abstract{We introduce and study a generalization $s_{(\mu|\lambda)}$ of the Schur functions called the almost symmetric Schur functions. These functions simultaneously generalize the finite variable key polynomials and the infinite variable Schur functions. They form a homogeneous basis for the space of almost symmetric functions and are defined using a family of recurrences involving the isobaric divided difference operators and limits of Weyl symmetrization operators. The $s_{(\mu|\lambda)}$ are the $q=t=0$ specialization of the stable limit non-symmetric Macdonald functions $\widetilde{E}_{(\mu|\lambda)}$ defined by the author in previous work. We find a combinatorial formula for these functions simultaneously generalizing well known formulas for the Schur functions and the key polynomials. Further, we prove positivity results for the coefficients of the almost symmetric Schur functions expanded into the monomial basis and into the monomial-Schur basis of the space of almost symmetric functions. The latter positivity result follows after realizing the almost symmetric Schur functions $s_{(\mu|\lambda)}$ as limits of characters of representations of parabolic subgroups in type $\GL.$}

\tableofcontents

\section{Introduction}

The Schur functions $s_{\lambda}$ are a central object in algebraic combinatorics  directly linking the representation theory of symmetric groups and $\GL_n$, the geometry of Grassmanians, and the combinatorics of Young tableaux. The functions $s_{\lambda}$ form an exceptional basis for the ring of symmetric functions $\Lambda$ with many remarkable combinatorial properties. In modern algebraic combinatorics, a good deal of work has been devoted to studying generalizations of the Schur functions. In particular, the symmetric Macdonald functions $P_{\lambda}[X;q,t]$ \cite{Macdonald} are an important two-parameter generalization of the Schur functions $s_{\lambda} = P_{\lambda}[X;0,0]$ relating to the representation theory of double affine Hecke algebras and the geometry of Hilbert schemes. The symmetric Macdonald functions are the symmetric analogues of the non-symmetric Macdonald polynomials $E_{\mu}$ of Cherednik \cite{C_2001}. These finite variable polynomials $E_{\mu}$ are important in the study of double affine Hecke algebras and satisfy many interesting algebraic and combinatorial properties \cite{haglund2007combinatorial}. Particularly, the $E_{\mu}$ are simultaneous eigenvectors for a special family of operators known as the \textit{Cherednik operators} $Y_i.$ It was shown by Ion \cite{Ion_2003} that the specialization $q,t \rightarrow 0$ of the non-symmetric Macdonald polynomials recovers the key polynomials $\mathcal{K}_{\mu}.$ The key polynomials are a non-symmetric analogue of the Schur functions relating to the geometry of Schubert varieties through the famous Demazure character formula \cite{Dem_1974} \cite{Andersen1985}.

Recently the proof of the Shuffle Theorem of Carlsson-Mellit \cite{CM_2015} introduced the idea of extending the ring of symmetric to include \textit{almost symmetric} functions; that is infinite variable polynomials $f(x_1,x_2,\ldots)$ for which there some $n \geq 1$ such that $f$ is symmetric in the variables $x_n,x_{n+1},\ldots.$ Geometrically, the almost symmetric functions roughly correspond to the equivariant $K$-theory of the parabolic flag Hilbert schemes of points in $\mathbb{C}^2$ \cite{GCM_2017}.  It was shown by Ion-Wu \cite{Ion_2022} that the ring of almost symmetric functions $\sP_{as}^{+}$ is a module for the positive stable-limit double affine Hecke algebra. This action includes a special family of commuting operators $\y_i$ known as the \textit{limit Cherednik operators}. It was shown in the author's prior work \cite{MBWArxiv} that there exists a family of almost symmetric functions $\widetilde{E}_{(\mu|\lambda)}[x_1,x_2,\ldots;q,t]$ called the \textit{stable-limit non-symmetric Macdonald functions} indexed by pairs of compositions $\mu$ and partitions $\lambda$ which are simultaneous eigenvectors for the limit Cherednik operators $\y_i.$ On the non-symmetric extreme $\widetilde{E}_{(\mu|\emptyset)}$ are limits of non-symmetric Macdonald polynomials $E_{\mu*0^n}$ and on the fully-symmetric extreme $\widetilde{E}_{(\emptyset|\lambda)}$ recover the symmetric Macdonald polynomials $P_{\lambda}$. These almost symmetric functions are related to a similar construction of Lapointe of \textit{m-symmetric} Macdonald polynomials \cite{lapointe2022msymmetric} and to the work of Goodberry \cite{Goodberry} and Orr-Goodberry \cite{goodberry2023geometric}. 

It is natural to wonder whether the $q,t \rightarrow 0$ specializations of the $\widetilde{E}_{(\mu|\lambda)}$ fill a role in the theory of almost symmetric functions similar to the special place that the Schur functions fill in the theory of symmetric functions. In this paper, we introduce a family of almost symmetric functions $s_{(\mu|\lambda)}$ called the \textit{almost symmetric Schur functions} (Definition \ref{almost sym schur def}). These functions are the $q,t \rightarrow 0$ specialization of the $\widetilde{E}_{(\mu|\lambda)}$ (Theorem \ref{specialization theorem}) and may be defined using a family of simple recurrence relations involving the \textit{isobaric divided difference operators} $\xi_i$ and the \textit{Weyl symmetrization operators} $W_k$. We find an explicit combinatorial model for the $s_{(\mu|\lambda)}$ derived from the Haglund-Haiman-Loehr formula for the non-symmetric Macdonald polynomials (Theorem \ref{combinatorial formula for almost sym schur}). Using this formula we find a non-negative combinatorial formula for the monomial expansion of the almost symmetric Schur functions (Theorem \ref{Positivity for almost symmetric Kostka coefficients}). We also realize the $s_{(\mu|\lambda)}$ as stable-limits of certain key polynomials (Proposition \ref{almost sym schur from key}). This immediately shows that the $s_{(\mu|\lambda)}$ are a $\mathbb{Q}$-basis for the space of almost symmetric functions with rational coefficients (Corollary \ref{almost sym schur are basis}). Using the Demazure character formula, we find a representation-theoretic interpretation for the $s_{(\mu|\lambda)}$ as limits of characters of certain representations $\mathcal{V}^{(n)}(\mu|\lambda)$ of parabolic subgroups in type $\GL$ (Lemma \ref{rep theory of almost sym schur}). From this we prove that the monomial-Schur expansions of the $s_{(\mu|\lambda)}$ have non-negative integer coefficients counting the multiplicities of certain irreducible representations for Levi subgroups in type $\GL$ in the modules $\mathcal{V}^{(n)}(\mu|\lambda)$ (Theorem \ref{rep theory for schur expansion}).

\subsection{Acknowledgements}

The author would like to thank their advisor Monica Vazirani for all of her continued invaluable guidance throughout the author's graduate school career at UC Davis. The author would also like to thank Nicolle Gonz\'{a}lez and Daniel Orr for helpful discussions about Demazure characters and Macdonald polynomials.

\section{Definitions and Notations}

\subsection{Basic Combinatorics}

\begin{defn}
 In this paper, a \textbf{\textit{composition}} will refer to a finite tuple $\mu = (\mu_1,\ldots,\mu_n)$ of non-negative integers. We allow for the empty composition $\emptyset$ with no parts. We will let $\Comp$ denote the set of all compositions. The length of a composition $\mu = (\mu_1,\ldots,\mu_n)$ is $\ell(\mu) = n$ and the size of the composition is $| \mu | = \mu_1+\ldots+\mu_n$. As a convention we will set $\ell(\emptyset) = 0$ and $|\emptyset| = 0.$ We say that a composition $\mu$ is \textbf{\textit{reduced}} if $\mu = \emptyset$ or $\mu_{\ell(\mu)} \neq 0.$ We will let $\Compred$ denote the set of all reduced compositions. Given two compositions $\mu = (\mu_1,\ldots,\mu_n)$ and $\beta = (\beta_1,\ldots,\beta_m)$, define $\mu * \beta = (\mu_1,\ldots,\mu_n,\beta_1,\ldots,\beta_m)$. A \textbf{\textit{partition}} is a composition $\lambda = (\lambda_1,\ldots,\lambda_n)$ with $\lambda_1\geq \ldots \geq \lambda_n \geq 1$. Note that vacuously we allow for the empty partition $\emptyset.$ We denote the set of all partitions by $\Par$. We denote by $\Sigma$ the set of all pairs $(\mu|\lambda)$ with $\mu \in \Compred$ and $\lambda \in \Par.$
 
 We denote by $\sort(\mu)$ the partition obtained by ordering the nonzero elements of $\mu$ in weakly decreasing order. We define $\rev(\mu)$ to be the composition obtained by reversing the order of the elements of $\mu$. 
 
 We will in a few instances use the notation $\mathbbm{1}(p)$ to denote the value $1$ if the statement p is true and $0$ otherwise.
 \end{defn}

\begin{defn}\label{symmetric group def}
    The symmetric group $\mathfrak{S}_n$ is defined as the set of bijective maps $\sigma: [n] \rightarrow [n]$ with multiplication given by function composition where $[n]:= \{1,\ldots, n\}$. For $1 \leq i \leq n-1$ we will write $s_i$ for the transposition swapping $i,i+1$ and fixing everything else. For $\sigma \in \mathfrak{S}_n$ the \textbf{\textit{length}} of $\sigma$, $\ell(\sigma)$, is defined to be the minimal number of $s_i$ required to express $\sigma$, i.e. $\sigma = s_{i_1}\cdots s_{i_r}.$ For any $\mu = (\mu_1,\ldots, \mu_r)$ with $\mu_i \geq 1$ and $\mu_1+\ldots + \mu_r = n$ we define the \textbf{\textit{Young subgroup}} $\mathfrak{S}_{\mu}$ to be the group generated by the $s_i$ with $i \in \{\mu_1+\ldots + \mu_{j-1} +1,\ldots,\mu_1+\ldots + \mu_{j-1}+\mu_{j} \}$ for some $0 \leq j \leq r.$
\end{defn}

We have the following alternative presentation of the symmetric group $\mathfrak{S}_n.$

\begin{prop}[Coxeter Presentation]
    The symmetric group $\mathfrak{S}_n$ is generated by elements $s_1,\ldots, s_{n-1}$ subject to the relations:
    \begin{itemize}
        \item $s_i^2 = 1$
        \item $s_{i}s_{i+1}s_i = s_{i+1}s_is_{i+1}$
        \item $s_is_j = s_j s_i$ for $|i-j| > 1.$
    \end{itemize}
\end{prop}

In line with the conventions in \cite{haglund2007combinatorial} we define the Bruhat order on the type $GL_n$ weight lattice $\mathbb{Z}^{n}$ as follows. 

\begin{defn}
Let $e_1,...,e_n$ be the standard basis of $\mathbb{Z}^n$ and let $\alpha \in \mathbb{Z}^n$. We define the\textbf{\textit{ Bruhat ordering}} on $\mathbb{Z}^n$, written simply by $<$, by first defining cover relations for the ordering and then taking their transitive closure. If $i<j$ such that $\alpha_i < \alpha_j$ then we say $\alpha > (ij)(\alpha)$ and additionally if $\alpha_j - \alpha_i > 1$ then $(ij)(\alpha) > \alpha + e_i - e_j$ where $(ij)$ denotes the transposition swapping $i$ and $j.$
\end{defn}

 \subsection{Polynomials}
 Throughout this paper the variables $q$ and $t$ are assumed to be commuting free variables.
 \begin{defn}
    Define $\sP_n:= \mathbb{Q}(q,t)[x_1^{\pm 1},\ldots, x_n^{\pm 1}]$ for the space of Laurent polynomials in $n$ variables over $\mathbb{Q}(q,t)$ and define $\sP_n^{+}:= \mathbb{Q}(q,t)[x_1,\ldots, x_n]$ for the subspace of polynomials. We define algebra homomorphisms $\Xi^{(n)}: \sP_{n+1}^{+} \rightarrow \sP_{n}^{+}$ by $$\Xi^{(n)}(x_1^{a_1}\cdots x_n^{a_n}x_{n+1}^{a_{n+1}}) = \mathbbm{1}(a_{n+1} = 0) x_1^{a_1}\cdots x_n^{a_n}.$$ The symmetric group $\mathfrak{S}_n$ acts naturally on $\sP_n$ by algebra automorphisms via 
    $$\sigma(f(x_1,\ldots, x_n)) = f(x_{\sigma(1)},\ldots, x_{\sigma(n)}).$$ 
\end{defn}

We will need to consider the action of the finite Hecke algebras in type $\GL$ on polynomials.

\begin{defn}\label{finite hecke alg defn}
    Define the finite Hecke algebra $\sH_n$ to be the $\mathbb{Q}(q,t)$-algebra generated by $T_1,\ldots, T_{n-1}$ subject to the relations 
    \begin{itemize}
        \item $(T_i-1)(T_i+t) = 0$ for $1\leq i \leq n-1$
        \item $T_iT_{i+1}T_1 = T_{i+1}T_iT_{i+1}$ for $1\leq i \leq n-2$
        \item $T_iT_j = T_jT_i$ for $|i-j| > 1.$
    \end{itemize}

    The polynomial representation of $\sH_n$ on $\sP_n^{+}$ is determined by the following action of the $T_i:$

    $$T_i(f):= s_i(f)+(1-t)x_i\frac{f-s_i(f)}{x_i-x_{i+1}}.$$
    
    For $0 \leq k \leq n$ let $\epsilon_k^{(n)} \in \sH_n$ denote the (normalized) trivial idempotent given by 
    $$\epsilon_k^{(n)}:= \frac{1}{[n-k]_{t}!}\sum_{\sigma \in \mathfrak{S}_{(1^k,n-k)}} t^{{n-k\choose 2} - \ell(\sigma)} T_{\sigma}.$$ Here
    $[m]_{t}!:= \prod_{i=1}^{m}(\frac{1-t^i}{1-t}).$ 
\end{defn}

We now review an important family of polynomials.

\begin{defn} \label{defn3}
The \textbf{\textit{non-symmetric Macdonald polynomials}} (for $GL_n$) are a family of Laurent polynomials $E_{\mu} \in \mathscr{P}_n$ for $\mu \in \mathbb{Z}^n$ uniquely determined by the following:

\begin{itemize}
    \item Triangularity: Each $E_{\mu}$ has a monomial expansion of the form $E_{\mu} = x^{\mu} + \sum_{\lambda < \mu} a_{\lambda}x^{\lambda}$

    \item Weight Vector: Each  $E_{\mu}$ is a weight vector for the \textbf{\textit{Cherednik operators}} $Y_1^{(n)},\ldots,Y_n^{(n)}$.
\end{itemize}
\end{defn}

We refer the reader to \cite{C_2001} for a review of DAHA and the Cherednik operators $Y_i.$ Importantly, the set $\{E_{\mu} | \mu \in \mathbb{Z}^n \}$ is a basis for $\sP_n$ with distinct $Y^{(n)}$ weights. For $\mu \in \mathbb{Z}^n$, $E_{\mu}$ is homogeneous with degree $\mu_1+\ldots +\mu_n$. Further, the set of $E_{\mu}$ corresponding to $\mu \in \mathbb{Z}_{\geq 0}^n$ gives a basis for $\mathscr{P}_n ^{+}$. 

\subsection{Combinatorial Formula for Non-symmetric Macdonald Polynomials}

In \cite{haglund2007combinatorial}, Haglund, Haiman, and Loehr give an explicit monomial expansion formula for the non-symmetric Macdonald polynomials in terms of the combinatorics of \textbf{\textit{non-attacking labellings}} of certain box diagrams corresponding to compositions which we will now review.

\begin{defn}\cite{haglund2007combinatorial} \label{HHL defn}
For a composition $\mu = (\mu_1,\dots,\mu_n)$ define the column diagram of $\mu$ as 
$$dg'(\mu):= \{(i,j)\in \mathbb{N}^2 : 1\leq i\leq n, 1\leq j \leq \mu_i \}.$$ This is represented by a collection of boxes in positions given by $dg'(\mu)$. The augmented diagram of $\mu$ is given by 
$$\widehat{dg}(\mu):= dg'(\mu)\cup\{(i,0): 1\leq i\leq n\}.$$
Visually, to get $\widehat{dg}(\mu)$ we are adding a bottom row of boxes on length $n$ below the diagram $dg'(\mu)$. Given $u = (i,j) \in dg'(\mu)$ define the following:
\begin{itemize}
    \item $\leg(u) := \{(i,j') \in dg'(\mu): j' > j\}$
    \item $\arm^{\mathleft}(u) := \{(i',j) \in dg'(\mu): i'<i, \mu_{i'} \leq \mu_i\} $
    \item $\arm^{\mathright}(u):= \{(i',j-1) \in \widehat{dg}(\mu): i'>i, \mu_{i'}<\mu_{i}\}$
    \item $\arm(u) := \arm^{\mathleft}(u) \cup \arm^{\mathright}(u)$
    \item $\mathleg(u):= |\leg(u)| = \mu_i -j$
    \item $a(u) := |\arm(u)|.$
\end{itemize}
A filling of $\mu$ is a function $\sigma: dg'(\mu) \rightarrow \{1,...,n\}$ and given a filling there is an associated augmented filling $\widehat{\sigma}: \widehat{dg}(\mu) \rightarrow \{1,...,n\}$ extending $\sigma$ with the additional bottom row boxes filled according to $\widehat{\sigma}((j,0)) = j$ for $j = 1,\dots,n$. Distinct lattice squares $u,v \in \mathbb{N}^2$ are said to attack each other if one of the following is true:
\begin{itemize}
\item $u$ and $v$ are in the same row 
\item $u$ and $v$ are in consecutive rows and the box in the lower row is to the right of the box in the upper row.
\end{itemize}
A filling $\sigma: dg'(\mu) \rightarrow \{1,\dots,n\}$ is non-attacking if $\widehat{\sigma}(u) \neq \widehat{\sigma}(v)$ for every pair of attacking boxes $u,v \in \widehat{dg}(\mu).$
For a box $u= (i,j)$ let $d(u) = (i,j-1)$ denote the box just below $u$. Given a filling $\sigma:dg'(\mu)\rightarrow \{1,\dots,n\}$, a descent of $\sigma$ is a box $u \in dg'(\mu)$ such that $\widehat{\sigma}(u) > \widehat{\sigma}(d(u)).$
Set $\Des(\widehat{\sigma})$ to be the set of descents of $\widehat{\sigma}$ and define 
$$\maj(\widehat{\sigma}):= \sum_{u \in \Des(\widehat{\sigma})} (\mathleg(u)+1).$$
The reading order on the diagram $\widehat{dg}(\mu)$ is the total ordering on the boxes of $\widehat{dg}(\mu)$ row by row, from top to bottom, and from right to left within each row. If $\sigma: dg'(\mu) \rightarrow \{1,\dots,n\}$ is a filling, an inversion of $\widehat{\sigma}$ is a pair of attacking boxes $u,v \in \widehat{dg}(\mu)$ such that $u < v$ in reading order and $\widehat{\sigma}(u) > \widehat{\sigma}(v).$ Set $\Inv(\widehat{\sigma})$ to be the set of inversions of $\widehat{\sigma}$. Define the statistics 
\begin{itemize}
    \item $\inv(\widehat{\sigma}):= |\Inv(\widehat{\sigma})| -|\{i<j: \mu_i \leq \mu_j\}| - \sum_{u \in \Des(\widehat{\sigma})} a(u)$
    \item $\coinv(\widehat{\sigma}):= \left( \sum_{u \in dg'(\mu)} a(u) \right) -\inv(\widehat{\sigma}).$
    \end{itemize}
Lastly, for a filling $\sigma:dg'(\mu) \rightarrow \{1,\dots,n\}$ set $$x^{\sigma}:= x_1^{|\sigma^{-1}(1)|}\cdots x_n^{|\sigma^{-1}(n)|}.$$
\end{defn}

The Haglund-Haiman-Loehr combinatorial formula for non-symmetric Macdonald polynomials can now be stated.

\begin{thm}\cite{haglund2007combinatorial} \label{HHL}
For a composition $\mu$ with $\ell(\mu) = n$ the following holds:
    $$E_{\mu} = \sum_{\substack{\sigma: \mu \rightarrow [n]\\ \text{non-attacking}}} x^{\sigma}q^{\maj(\widehat{\sigma})}t^{\coinv(\widehat{\sigma})} \prod_{\substack{u \in dg'(\mu) \\ \widehat{\sigma}(u) \neq \widehat{\sigma}(d(u))}} \left( \frac{1-t}{1-q^{\mathleg(u)+1}t^{a(u)+1}} \right). $$
\end{thm}

We may better understand the statistic $\coinv$ through the next definition.

\begin{defn}\cite{haglund2007combinatorial}
    Let $\sigma: \mu \rightarrow [n]$ be a non-attacking labelling. A \textbf{\textit{co-inversion triple}} is a triple of boxes $(u,v,w)$ in the diagram $\widehat{dg}(\mu)$ of one of the following two types
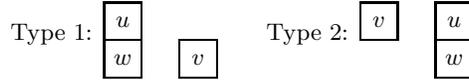
\begin{center}
    Type 1: \begin{ytableau}
        u & \none & \none \\
        w &   \none   & v \\
        \end{ytableau} ~~~~ Type 2: \begin{ytableau}
        v & \none & u \\
        \none &   \none   & w \\
        \end{ytableau}
\end{center}

that satisfy the following criteria:
\begin{itemize}
    \item in Type 1 the column containing $u$ and $w$ is strictly taller than the column containing $v$
    \item in Type 2 the column containing $u$ and $w$ is weakly taller than the column containing $v$
    \item in either Type 1 or Type 2 $\widehat{\sigma}(u)< \widehat{\sigma}(v)< \widehat{\sigma}(w)$ or $\widehat{\sigma}(v)< \widehat{\sigma}(w)< \widehat{\sigma}(u)$ or $\widehat{\sigma}(w)< \widehat{\sigma}(u)< \widehat{\sigma}(v).$
\end{itemize}

Informally, in Type 1 we require the entries to strictly increase clockwise and in Type 2 we require the entries to strictly increase counterclockwise.
\end{defn}

Co-inversion triples are important because they have the same count as the complicated $\coinv$ statistic from Definition \ref{HHL defn}.

\begin{lem}\cite{haglund2007combinatorial}
    For a non-attacking labelling $\sigma: \mu \rightarrow [n]$, 
    $\coinv(\widehat{\sigma})$ equals the number of co-inversion triples of $\widehat{\sigma}.$
\end{lem}

\begin{example}
We finish this subsection with a visual example of a non-attacking filling and its associated statistics. Below is the augmented filling $\widehat{\sigma}$ of a non-attacking filling $\sigma: (3,2,0,1,0,0) \rightarrow [6]$ pictured as labels inside the boxes of $\widehat{dg}(3,2,0,1,0,0).$

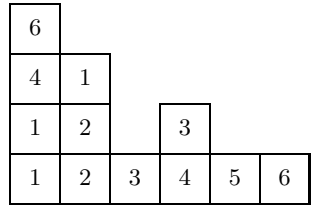
\begin{center}

\ytableausetup{centertableaux, boxframe= normal, boxsize= 2em}
\begin{ytableau}
6 & \none & \none & \none & \none & \none \\
4 &   1   & \none & \none & \none & \none \\
1 &   2   & \none &   3   & \none & \none \\
1 &   2   &   3   &   4   &   5   &   6   \\
\end{ytableau}

\end{center}

Let $u$ be the column 1 box of $\widehat{dg}(3,2,0,1,0,0)$ filled with a $4$ in the above diagram. Notice that $u$ is a descent box of $\widehat{\sigma}$ as $4$ is larger than the label $1$ of the box $d(u)$ just below $u.$ Further, we see that $a(u) = 2$ and $\mathleg(u) = 1$. Considering the diagram as a whole now we see that $x^{\sigma} = x_1^{2}x_2x_3x_4x_6$, $\maj(\widehat{\sigma}) = 3$, $|\Inv(\widehat{\sigma})| = 21$, $\inv(\widehat{\sigma}) = 14$, and $\coinv(\widehat{\sigma}) = 1.$ The contribution of this non-attacking labelling to the HHL formula for $E_{(3,2,0,1,0,0)} \in \sP_{6}^{+}$ is 
$$x_1^2x_2x_3x_4x_6 q^{3}t^{1} \left( \frac{1-t}{1-q^{1}t^3} \right)\left( \frac{1-t}{1-q^{1}t^2} \right)\left( \frac{1-t}{1-q^{2}t^{3}} \right)\left( \frac{1-t}{1-q^{1}t^2} \right).$$
\end{example}

\subsection{Symmetric Functions}

 \begin{defn}
 Define the \textbf{\textit{ring of symmetric functions}} $\Lambda$ to be the subalgebra of the inverse limit of the symmetric polynomial rings $\mathbb{Q}(q,t)[x_1,\ldots,x_n]^{\mathfrak{S}_n}$ with respect to the quotient maps $\Xi^{(n)}$ consisting of those elements with bounded $x$-degree. For $i \geq 1$ define the $i$-th \textbf{\textit{power sum symmetric function}} by
 $$p_i = x_1^i + x_2^i + \dots .$$ It is a classical result that $\Lambda$ is isomorphic to $\mathbb{Q}(q,t)[p_1,p_2,\dots]$. For any expression $G = a_1g^{\mu_1} + a_2g^{\mu_2} +\dots$ with rational scalars $a_i \in \mathbb{Q}$ and distinct monomials $g^{\mu_i}$ in a set of algebraically independent commuting free variables $\{g_1,g_2,\dots\}$ the \textbf{\textit{plethystic evaluation}} of $p_i$ at the expression $G$ is defined to be $$p_i[G] := a_1g^{i\mu_1} +a_2g^{i \mu_2}+ \dots .$$ Note that $g_i$ are allowed to be $q$ or $t.$ Here we are using the convention that $i\mu = (i\mu_1,\ldots, i \mu_r)$ for $\mu = (\mu_1,\cdots, \mu_r).$ The definition of plethystic evaluation on power sum symmetric functions extends to all symmetric functions $F \in \Lambda$ by requiring $F \rightarrow F[G]$ be a $\mathbb{Q}(q,t)$-algebra homomorphism. Note that for $F \in \Lambda$, $F = F[x_1+x_2+\ldots]$ and so we will often write $F = F[X]$ where $X:= x_1+ x_2 + \ldots.$ For a partition $\lambda$ define the \textbf{\textit{monomial symmetric function}} $m_\lambda$ by 
 $$m_\lambda := \sum_{\mu} x^{\mu}$$
 where we range over all distinct monomials $x^{\mu}$ such that $\sigma(\mu) = \lambda$ for some permutation $\sigma$. For $n \geq 0$ define the \textbf{\textit{complete homogeneous symmetric function}} $h_n$ by 
 $$h_n:= \sum_{|\lambda|= n} m_\lambda .$$ For $\lambda \in \Par$ the Schur functions are given as 
 $$s_{\lambda} = \sum_{T \in \SSYT(\lambda)} x^{T}$$ where $\SSYT(\lambda)$ denotes the set of semi-standard Young tableaux of shape $\lambda.$
 We can extend plethysm to $\mathbb{Q}(q,t)[[p_1,p_2,\dots]]$.  The \textbf{\textit{plethystic exponential}} is defined to be the element of $\mathbb{Q}(q,t)[[p_1,p_2,\dots]]$ given by 
 $$\Exp[X]:= \sum_{n \geq 0} h_n[X].$$ 
 \end{defn}

We will need to consider the following operators.

\begin{defn}\label{jing vertex op}
    For $n \geq 0$ define the \textbf{\textit{Jing vertex operator}} $\mathscr{B}_n \in \End_{\mathbb{Q}(q,t)}(\Lambda)$ by
    $$\mathscr{B}_n[F] : = \langle z^n \rangle F[X-z^{-1}]\Exp[(1-t)zX].$$
    Here $\langle z^n \rangle $ denotes the operator which extracts the coefficient of $z^n$ of any Laurent series in $z$.
\end{defn}

\subsection{Almost Symmetric Functions}

The following ring was studied by Ion-Wu in their paper \cite{Ion_2022}.

\begin{defn}\cite{Ion_2022} \label{defn5}
  For $k \geq 0$ define the ring $\mathscr{P}(k)^{+} := \mathbb{Q}(q,t)[x_1,\ldots,x_k]\otimes \Lambda[\mathfrak{X}_{k}]$ where $\mathfrak{X}_{k}:= x_{k+1}+x_{k+2}+\ldots.$ The \textbf{\textit{ring of almost symmetric functions}} is given by $\mathscr{P}_{as}^{+} := \bigcup_{k\geq 0} \mathscr{P}(k)^{+}$. 
 \end{defn}

The ring $\sP_{as}^{+}$ is a free graded $\Lambda$-module with homogeneous basis given simply by the set of monomials $x^{\mu}$ with $\mu$ reduced. Therefore, $\sP_{as}^{+}$ has the homogeneous $\mathbb{Q}(q,t)$ basis given by all $x^{\mu}m_{\lambda}[X]$ ranging over $(\mu|\lambda) \in \Sigma$. Further, the dimension of the homogeneous degree d part of $\mathscr{P}(k)^{+}$ is equal to the number of pairs $(\mu|\lambda) \in \Sigma$ with $|\mu|+|\lambda| = d$ and $\ell(\mu) \leq k$.

We will need to consider a family of symmetrization operators on $\sP_{as}^{+}.$

\begin{lem}\label{idempotent lemma}\cite{MBWArxiv}
    For any $k \geq 0$ the sequence of partial symmetrization operators $( \epsilon_k^{(n)})_{n \geq k}$ converges (in the sense of Ion-Wu \cite{Ion_2022}) to a well defined map $\epsilon_k: \sP_{as}^{+} \rightarrow \sP(k)^{+}.$ These maps satisfy the following properties:
\begin{itemize}
        \item $\epsilon_k^{2} = \epsilon_k$
        \item $\epsilon_k T_i = T_i \epsilon_k = \epsilon_k$ for $ i \geq k+1$
        \item $T_i\epsilon_k = \epsilon_kT_i$ for $1 \leq i \leq k-1$
        \item $\epsilon_k \epsilon_{\ell} = \epsilon_{\min(k,\ell)}$
        \item For any $a_1,\ldots, a_{k+1} \geq 0$ and $F \in \Lambda$
        $$\epsilon_k(x_1^{a_1}\cdots x_k^{a_k}x_{k+1}^{a_{k+1}}F[\mathfrak{X}_{k+1}]) = x_1^{a_1}\cdots x_k^{a_k} \mathscr{B}_{a_{k+1}}(F)[\mathfrak{X}_k].$$
    \end{itemize}
\end{lem}

In previous work \cite{MBWArxiv} the author defined the following almost symmetric version of the non-symmetric Macdonald polynomials.

\begin{thm}\label{stable-limit non-sym MacD function defn}\cite{MBWArxiv}
For $(\mu|\lambda) \in \Sigma$ the limit
    $$\widetilde{E}_{(\mu|\lambda)}[x_1,x_2,\ldots;q,t]:= \lim_{n} \epsilon_{\ell(\mu)}^{(n)}\left(E_{\mu*\lambda*0^{n-(\ell(\mu)+\ell(\lambda))}}(x_1,\ldots,x_n;q,t)\right) $$
    exists (in the sense of Ion-Wu \cite{Ion_2022}). The set of $\{ \widetilde{E}_{(\mu|\lambda)}\}_{(\mu|\lambda) \in \Sigma}$ is a homogeneous basis for $\sP_{as}^{+}.$
\end{thm}

\begin{remark}
     $\{\widetilde{E}_{(\mu|\lambda)}| (\mu|\lambda) \in \Sigma \}$ is a weight basis of $\sP_{as}^{+}$ for the limit Cherednik operators $\y$ \cite{MBWArxiv}. Each $\widetilde{E}_{(\mu|\lambda)}$ is homogeneous of degree $|\mu| + |\lambda|.$
\end{remark}

The stable-limit non-symmetric Macdonald functions satisfy the following recursion.

\begin{prop}\label{macd recursion}\cite{MBWArxiv}
    \begin{itemize}
        \item If $\mu \in \Compred$
        $$\widetilde{E}_{(\mu|\emptyset)} = \lim_{n} E_{\mu*0^n}$$ where the limit converges in the sense of Ion-Wu \cite{Ion_2022}.
        \item If $\mu_{i} > \mu_{i+1}$ and $s_i(\mu) \in \Compred$ then $$\widetilde{E}_{(s_i(\mu)|\lambda)} = 
  \left( T_i + \frac{(1-t)\widetilde{\alpha}_{(\mu|\lambda)}(i+1)}{\widetilde{\alpha}_{(\mu|\lambda)}(i)- \widetilde{\alpha}_{(\mu|\lambda)}(i+1)} \right)  \widetilde{E}_{(\mu|\lambda)}$$ where $\widetilde{\alpha}_{(\mu|\lambda)}$ denotes the $\y$ weight of $\widetilde{E}_{(\mu|\lambda)}.$
        \item Whenever $\mu_r \geq \lambda_1$ and $\mu_{r-1} \neq 0,$
        $$ \epsilon_{r-1} \left( \widetilde{E}_{(\mu_1,\dots,\mu_r|\lambda_1,\dots,\lambda_k)} \right) = \widetilde{E}_{(\mu_1,\dots,\mu_{r-1}|\mu_r,\lambda_1,\dots,\lambda_k)}.$$ 
    \end{itemize}
\end{prop}

We have the following stable-limit variant of the HHL formula \ref{HHL}.

\begin{thm}\cite{MBWArxiv}\label{convergence of macdonald}

For $\mu \in \Compred$
$$ \widetilde{E}_{(\mu|\emptyset)} = \sum_{\substack{\lambda ~ \text{partition}  \\ |\lambda| \leq |\mu|}} m_{\lambda}[x_{n+1}+\ldots] \sum_{\substack{\sigma:\mu * 0^{\ell(\lambda)} \rightarrow [n+\ell(\lambda)]\\ \text{non-attacking} \\ \forall i = 1,...,\ell(\lambda) \\ \lambda_i = |\sigma^{-1}(n+i)|}} x_1 ^{|\sigma^{-1}(1)|}\cdots x_n ^{|\sigma^{-1}(n)|} \widetilde{\Gamma}(\widehat{\sigma})$$  

where 

$$\widetilde{\Gamma}(\widehat{\sigma}) : = q^{\maj(\widehat{\sigma})}t^{\coinv(\widehat{\sigma})} \prod_{\substack{ u \in dg'(\mu * 0^{\ell(\lambda)}) \\ \widehat{\sigma}(u) \neq \widehat{\sigma}(d(u)) \\ u ~ \text{not in row } 1 }} \left( \frac{1-t}{1-q^{\mathleg(u)+1}t^{a(u)+1}} \right) \prod_{\substack{ u \in dg'(\mu * 0^{\ell(\lambda)}) \\ \widehat{\sigma}(u) \neq \widehat{\sigma}(d(u)) \\ u ~ \text{in row } 1 }} \left( 1-t \right). $$

\end{thm}

\section{Symmetrization and Specialization}

In the next section we build some tools which will allow us to compute and study the specializations of the stable-limit non-symmetric Macdonald functions at $q = t = 0.$

\subsection{Weyl Symmetrization and Isobaric Divided Difference Operators}

We now recall the definition of the Weyl symmetrization map and its partial symmetrization analogues. Informally, these maps are the $t = 0$ specialization of the $\epsilon_{k}^{(n)}$ operators defined previously.

\begin{defn}
    Let $ 0\leq k \leq n$. We define the \textbf{\textit{partial Weyl symmetrizer}}, $W_k^{(n)}$, to be the map $$W_k^{(n)}: \sP_n^{+} \rightarrow (\sP_n^{+})^{\mathfrak{S}_{(1^k,n-k)}}$$ given by 
    $$W_k^{(n)}(f(x_1,\ldots, x_n)):= \sum_{\sigma \in \mathfrak{S}_{(1^k,n-k)}} \sigma \left( f(x_1,\ldots,x_n) \prod_{k+1\leq i< j \leq n} \left( \frac{1}{1-x_j/x_i} \right) \right).$$ 
\end{defn}

\begin{remark}
    
Notice that these maps are defined over $\mathbb{Q}$ (over $\mathbb{Z}$ in fact). We may rewrite the given definition of $W_k^{(n)}$ as 
$$W_k^{(n)}(f(x_1,\ldots, x_n)) = \frac{ \sum_{\sigma \in \mathfrak{S}_{(1^k,n-k)}} (-1)^{\ell(\sigma)} \sigma \left( x^{\delta_k^{(n)}}f(x)  \right)}{\prod_{k+1\leq i< j \leq n} \left( x_i-x_j \right)}$$ where $\delta_{k}^{(n)}:= 0^{k}*(n-k-1,\ldots,1,0).$
 \end{remark}

 We will need a few properties of the Weyl symmetrization operators.

\begin{lem}
    As elements of $\End_{\mathbb{Q}(q,t)}(\sP_n^{+})$ the operators $W_{k}^{(n)}$ satisfy the following:
    \begin{itemize}
        \item $(W_{k}^{(n)})^2 = W_{k}^{(n)}$
        \item $\sigma W_{k}^{(n)} =  W_{k}^{(n)} \sigma$ for $ \sigma \in \mathfrak{S}_{(k,n-k)}$
        \item $ \sigma W_{k}^{(n)}  = W_{k}^{(n)}$ for $\sigma \in \mathfrak{S}_{(1^k,n-k)}$
        \item $W_k^{(n)}W_j^{(n)} = W_k^{(n)}$ for $k < j.$
    \end{itemize}
\end{lem}
\begin{proof}
    This will follow from Lemma \ref{specializations of Hecke maps} proven later in this section.
\end{proof}

\begin{lem}
   For $0 \leq k \leq n$
   $$W_{k}^{(n)}\Xi^{(n)} = \Xi^{(n)} W_{k}^{(n+1)}.$$
\end{lem}
\begin{proof}
This proof is standard and we leave its proof to the reader.        
\end{proof}

The above lemma allows for the following definition.

\begin{defn}
    Let $k \geq 0$ define the operator $W_k$ on $\sP_{as}^{+}$ as 
    $$W_k:= \lim_{n} W_k^{(n)}.$$
\end{defn}

As we will prove later, the operators $W_k$ are the $t = 0$ specializations of the partial Hecke symmetrizers $\epsilon_k.$

\begin{defn}
     Define the \textbf{\textit{isobaric divided difference operators}}, $\xi_1, \xi_2, \xi_3, \ldots$, on $\sP_{as}^{+}$ by 
    $$\xi_i(f) = \frac{x_if - x_{i+1}s_i(f)}{x_i-x_{i+1}}.$$
\end{defn}

\begin{lem}
We have the following relations for $i,j \geq 1:$
    \begin{itemize}
        \item $\xi_i^2 = \xi_i$
        \item $\xi_i\xi_{i+1}\xi_i = \xi_{i+1}\xi_i \xi_{i+1}$
        \item $\xi_i\xi_j = \xi_j \xi_i$ for $|i-j| >1.$
    \end{itemize}
\end{lem}
\begin{proof}
    This result is standard but in particular will follow from Lemma \ref{specializations of Hecke maps} proven later in this section.
\end{proof}

The above are the relations for the \textbf{\textit{0-Hecke algebra}}. The following standard lemma relates the Weyl symmetrizers $W_k^{(n)}$ to the isobaric divided difference operators $\xi_i.$

\begin{lem}
We have the recursion relation:
    $$W_k^{(n)} = \xi_{n-1}\cdots \xi_{k+1}W_{k+1}^{(n)}.$$ 
\end{lem}

One of the main utilities for defining the maps $W_k^{(n)}$ is that they generate the Schur polynomials in the following way.

\begin{prop}[Weyl Character Formula for $GL_n$]
   For $ \lambda \in \Par$ and $ \ell(\lambda) \leq n$
   $$W_0^{(n)}(x^{\lambda}) = s_{\lambda}(x_1,\ldots, x_n).$$
\end{prop}

 Here we review some relevant information about the key polynomials.

\begin{defn}\cite{Dem_1974}
    Let $n \geq 1$. Define the \textbf{\textit{key polynomials}} to be the unique collection of polynomials $\{\mathcal{K}_{\alpha}(x_1,\ldots, x_n) \}_{\alpha \in \mathbb{Z}_{\geq 0}^{n}} \subset \sP_n^{+}$ determined by the following properties:
    \begin{itemize}
        \item If $\alpha_1 \geq \ldots \geq \alpha_n$ then 
        $$\mathcal{K}_{\alpha}(x_1,\ldots, x_n)= x^{\alpha}.$$
        \item Whenever $\alpha_i > \alpha_{i+1}$
        $$\mathcal{K}_{s_i(\alpha)}(x_1,\ldots, x_n) = \xi_{i}(\mathcal{K}_{\alpha}(x_1,\ldots, x_n)).$$
    \end{itemize}
\end{defn}

We refer the reader to Kirillov \cite{Kirillov_2016} and Mason \cite{mason_2009} for an overview of key polynomials.

By a simple inductive argument we see that for $\alpha \in \mathbb{Z}_{\geq 0}^{n}$
$$\mathcal{K}_{(\alpha_1,\ldots, \alpha_n,0)}(x_1,\ldots, x_n,x_{n+1}) = \mathcal{K}_{(\alpha_1,\ldots, \alpha_n)}(x_1,\ldots, x_n).$$ As such we will refer to $\mathcal{K}_{\mu}(x)$ for $\mu \in \Compred$ unambiguously as an element of 
$\mathbb{Z}_{\geq 0}[x_1, x_2, x_3,\ldots] \subset \sP_{as}^{+}.$

\begin{remark}
    It is known that the key polynomials $\{ \mathcal{K}_{\alpha}| \alpha \in \mathbb{Z}_{\geq 0}^n \}$ form a basis for $\sP_n^{+}.$
    
    For $\lambda \in \Par$ and $ n \geq \ell(\lambda)$
    $$\mathcal{K}_{0^{n-\ell(\lambda)}*\rev(\lambda)}(x_1,\ldots,x_n) = s_{\lambda}(x_1,\ldots, x_n).$$ Further, if $\alpha = (\alpha_1,\ldots, \alpha_n) \in \mathbb{Z}_{\geq 0}^{n}$ and there exists some $1 \leq i < j \leq n$ with $\alpha_i < \ldots < \alpha_j$ then $\mathcal{K}_{\alpha}(x_1,\ldots, x_n)$ is symmetric in the variables $x_i,\ldots, x_j.$ In particular, for any $i \leq k \leq j-1$, $\xi_k(\mathcal{K}_{\alpha}(x_1,\ldots, x_n)) = \mathcal{K}_{\alpha}(x_1,\ldots, x_n).$
\end{remark}

\subsection{Specialization at $q=t=0$}

\begin{defn}\label{specialization definition}
    Define $\mathcal{O} \subset \sP_{as}^{+}$ to be the set of $f(x) \in \sP_{as}^{+}$ such that 
    $$f(x) = f(x_1,x_2,\ldots ;q,t) = \sum_{i} c^{(i)} x^{\mu^{(i)}}m_{\lambda^{(i)}}[X]$$ for some scalars $c^{(i)} = c^{(i)}(q,t) \in \mathbb{Q}[q][[t]] \cap \mathbb{Q}(q,t)$, $(\mu^{(i)}|\lambda^{(i)}) \in \Sigma.$ 
    
    Let $\sP_{as,\mathbb{Q}}^{+}$ denote the set of $f(x) \in \sP_{as}^{+}$ such that $$f(x) = \sum_{i} c^{(i)} x^{\mu^{(i)}}m_{\lambda^{(i)}}[X]$$ for some scalars $c^{(i)} \in \mathbb{Q}$, $(\mu^{(i)}|\lambda^{(i)}) \in \Sigma.$ Define the $\mathbb{Q}$-algebra homomorphism $\Upsilon:\mathcal{O} \rightarrow \sP_{as,\mathbb{Q}}^{+}$ by 
    $$\Upsilon(f(x_1,x_2,\ldots ;q,t)) := f(x_1,x_2,\ldots; 0,0).$$ Equivalently, 
    $$\Upsilon(f) := \lim_{q \rightarrow 0} \lim_{ t \rightarrow 0} f.$$
\end{defn}

\begin{remark}
    For arbitrary elements of $\mathbb{Q}(q,t)$ the limits $\lim_{q \rightarrow 0}$ and $\lim_{t \rightarrow 0}$ do not commute. For example, $\lim_{q \rightarrow 0} \lim_{ t \rightarrow 0} \frac{q}{t} $ does not exist whereas $\lim_{t \rightarrow 0} \lim_{ q \rightarrow 0} \frac{q}{t} = 0.$ However, these limits commute on $\mathbb{Q}[q][[t]] \cap \mathbb{Q}(q,t)$ so there is no ambiguity in Definition \ref{specialization definition}.
\end{remark}

We will need the following lemma.

\begin{lem}\label{continuity of specialization}
    Let $f_n \in \sP_n^{+}\cap \mathcal{O}$ with $\lim_{n} f_n = f \in \sP_{as}^{+}.$ Then $f \in \mathcal{O}$ and 
    $$\Upsilon(f) = \lim_{n} \Upsilon(f_n).$$ 
\end{lem}
\begin{proof}
    By the definition of Ion-Wu's convergence (see \cite{Ion_2022}) we know that we have for all $n \geq 1$
    $$f_n = \sum_{i=1}^{N} c_n^{(i)} x^{\mu^{(i)}}m_{\lambda^{(i)}}[x_1+\ldots+  x_n]$$ where 
    $c_n^{(i)} \in \mathbb{Q}(q,t)$, $(\mu^{(i)}|\lambda^{(i)}) \in \Sigma$ with $\lim_{n} c_n^{(i)} = c^{(i)} \in \mathbb{Q}(q,t)$ convergent $t$-adically. Since $f_n \in \sP_n^{+}\cap \mathcal{O}$ we know that $c_n^{(i)} = c_n^{(i)}(q,t) \in \mathbb{Q}[q][[t]] \cap \mathbb{Q}(q,t).$ Since $\mathbb{Q}[q][[t]]$ is topologically complete $t$-adically we must have $c^{(i)} \in \mathbb{Q}[q][[t]] \cap \mathbb{Q}(q,t).$ Then 
    it is clear that 
    $$ f = \sum_{i=1}^{N} c^{(i)} x^{\mu^{(i)}}m_{\lambda^{(i)}}[X] \in \mathcal{O}.$$ 
    A simple topological argument shows that $$\lim_{q \rightarrow 0} \lim_{ t \rightarrow 0} c^{(i)}(q,t) = \lim_{n} \lim_{q \rightarrow 0} \lim_{ t \rightarrow 0} c_n^{(i)}(q,t).$$ Then we find 
    \begin{align*}
        &\lim_{n} \Upsilon(f_n) \\
        &= \lim_{n} \Upsilon \left( \sum_{i=1}^{N} c_n^{(i)} x^{\mu^{(i)}}m_{\lambda^{(i)}}[x_1+\ldots+  x_n] \right)\\
        &= \lim_{n} \lim_{q \rightarrow 0} \lim_{ t \rightarrow 0} \sum_{i=1}^{N} c_n^{(i)} x^{\mu^{(i)}}m_{\lambda^{(i)}}[x_1+\ldots+  x_n]\\
        &= \lim_{n}  \sum_{i=1}^{N} \left( \lim_{q \rightarrow 0} \lim_{ t \rightarrow 0} c_n^{(i)} \right) x^{\mu^{(i)}}m_{\lambda^{(i)}}[x_1+\ldots+  x_n]\\
        &= \sum_{i=1}^{N}  \left(\lim_{q \rightarrow 0} \lim_{ t \rightarrow 0} c^{(i)} \right) x^{\mu^{(i)}}m_{\lambda^{(i)}}[X]\\
        &= \Upsilon(f).\\
    \end{align*}
\end{proof}

Here we recall a result of Ion \cite{Ion_2003} relating the non-symmetric Macdonald polynomials to the key polynomials.

\begin{thm}\cite{Ion_2003} \label{specialization of non sym Macd}
    For $\alpha \in \mathbb{Z}_{\geq 0}^{n}$
    $$\Upsilon(E_{\alpha}) = \mathcal{K}_{\alpha}.$$
\end{thm}

From Ion's result we find a known combinatorial formula for the key polynomials using the HHL combinatorial formula (see \ref{HHL}) for the non-symmetric Macdonald polynomials. For $\alpha \in \mathbb{Z}_{\geq 0}^{n}$ denote by $\mathcal{L}(\alpha)$ the set of non-attacking labellings $\sigma: \alpha \rightarrow [n]$ such that $\maj(\widehat{\sigma}) = \coinv(\widehat{\sigma}) = 0.$

\begin{prop}
    For $\alpha \in \mathbb{Z}_{\geq 0}^{n}$, 
    $$\mathcal{K}_{\alpha} = \sum_{\sigma \in \mathcal{L}(\alpha) } x^{\sigma}.$$
\end{prop}
\begin{proof}
    From the combinatorial formula for $E_{\alpha}$ (Theorem \ref{HHL}) we see that 
    $$E_{\alpha} = \sum_{\substack{\sigma: \alpha \rightarrow [n]\\ \text{non-attacking}}} x^{\sigma}q^{\maj(\widehat{\sigma})}t^{\coinv(\widehat{\sigma})} \prod_{\substack{u \in dg'(\alpha) \\ \widehat{\sigma}(u) \neq \widehat{\sigma}(d(u))}} \left( \frac{1-t}{1-q^{\mathleg(u)+1}t^{a(u)+1}} \right). $$
    Note that the values $\leg(u)$ and $\arm(u)$ are both non-negative so that $E_{\alpha} \in \mathcal{O}.$ Therefore, when we specialize $q \rightarrow \infty$ and $t \rightarrow 0$ we find that 
    $$\lim_{q \rightarrow 0} \lim_{ t \rightarrow 0} q^{\maj(\widehat{\sigma})}t^{\coinv(\widehat{\sigma})} \prod_{\substack{u \in dg'(\alpha) \\ \widehat{\sigma}(u) \neq \widehat{\sigma}(d(u))}} \left( \frac{1-t}{1-q^{\mathleg(u)+1}t^{a(u)+1}} \right) = \mathbbm{1}\left(\maj(\widehat{\sigma}) = \coinv(\widehat{\sigma}) = 0\right).$$
    Hence, from Theorem \ref{specialization of non sym Macd}
    $$\mathcal{K}_{\alpha} = \Upsilon(E_{\alpha}) = \sum_{\substack{\sigma: \alpha \rightarrow [n]\\ \text{non-attacking}\\ \maj(\widehat{\sigma}) = 0 \\ \coinv(\widehat{\sigma}) = 0 }} x^{\sigma} = \sum_{\sigma \in \mathcal{L}(\alpha) } x^{\sigma}.$$
\end{proof}

\begin{remark}
    Note that $\maj(\widehat{\sigma}) = 0$ is equivalent to $\Des(\widehat{\sigma}) = \emptyset$ which in turn is equivalent to $\widehat{\sigma}(u) \leq \widehat{\sigma}(d(u))$ i.e. $\widehat{\sigma}$ is weakly decreasing upwards along columns. The requirement that $\coinv(\widehat{\sigma}) = 0$ is equivalent to the statement that $\widehat{\sigma}$ has no co-inversion triples. Thus a non-attacking filling $\sigma$ is in $\mathcal{L}(\alpha)$ if $\widehat{\sigma}$ is weakly decreasing upwards along columns and has no co-inversion triples.
\end{remark}

As an easy application of Ion's result we may compute the specializations of all $\widetilde{E}_{(\mu|\emptyset)}.$

\begin{prop}\label{special case of specialization thm}
    For all $\mu \in \Compred$, $\widetilde{E}_{(\mu|\emptyset)} \in \mathcal{O}$ and
    $$\Upsilon(\widetilde{E}_{(\mu|\emptyset)}) = \mathcal{K}_{\mu}.$$
\end{prop}
\begin{proof}
Let $\mu \in \Compred.$ From the combinatorial formula Corollary \ref{convergence of macdonald} we may observe directly that $\widetilde{E}_{(\mu|\emptyset)} \in \mathcal{O}$. To see this note that each of the scalar coefficients of the expansion of $\widetilde{E}_{(\mu|\emptyset)}$ has the form 
$$q^{a}t^{b}\prod_{i}\left( \frac{1-t}{1-q^{c_{i}}t^{d_{i}}} \right)$$ for some $a,b,c_i,d_i \geq 0.$ By expanding the denominators 
$$\frac{1}{1-q^{c_{i}}t^{d_{i}}} = \sum_{m \geq 0}q^{mc_{i}}t^{md_{i}}$$ 
we see that 
$$q^{a}t^{b}\prod_{i}\left( \frac{1-t}{1-q^{c_{i}}t^{d_{i}}} \right) \in \mathbb{Q}[q][[t]]$$ as required.

As $\Upsilon(\widetilde{E}_{\mu})$ is now well defined, we may compute directly using Lemma \ref{continuity of specialization} to find
    \begin{align*}
    &\Upsilon(\widetilde{E}_{\mu})\\
    &= \lim_{n} \Upsilon(E_{\mu*0^n})\\
    &= \lim_{n} \mathcal{K}_{\mu*0^n}\\
    &= \lim_{n} \mathcal{K}_{\mu} \\
    &= \mathcal{K}_{\mu}.\\
    \end{align*}
\end{proof}

In the next lemma we will formalize the notion that the operators $\xi_i,W_k$ are the $q = t = 0$ specializations of $T_i, \epsilon_k$ respectively. This result is standard but we will include its proof for the sake of completeness. 

\begin{lem}\label{specializations of Hecke maps}
    For all $k \geq 0$ and $i \geq 1$,
    $\Upsilon \circ T_i|_{\mathcal{O}} = \xi_{i} \circ \Upsilon|_{\mathcal{O}}$
    and $\Upsilon \circ \epsilon_k|_{\mathcal{O}} = W_{k}\circ \Upsilon|_{\mathcal{O}}.$
    
\end{lem}
\begin{proof}
    Let $f= f(x;q,t) \in \sP_{as}^{+}\cap \mathcal{O}.$ Let $i \geq 1$ and $k \geq 0.$

    First, we have 
    \begin{align*}
        &\Upsilon \circ T_i(f) \\
        &=\Upsilon \left(s_i(f) + (1-t)x_i\frac{f-s_i(f)}{x_i-x_{i+1}} \right) \\
        &= s_i\Upsilon(f) + (1-0) x_i\frac{\Upsilon(f)-s_i\Upsilon(f)}{x_i-x_{i+1}}\\
        &= \left(s_i +x_i\frac{1-s_i}{x_i-x_{i+1}}\right)f(x;0,0) \\
        &= \left( \frac{(x_i-x_{i+1})s_i + x_i(1-s_i)}{x_i-x_{i+1}} \right) f(x;0,0) \\
        &= \left( \frac{x_i - x_{i+1}s_i}{x_i-x_{i+1}} \right) f(x;0,0) \\
        &= \xi_{i} f(x;0,0) \\
        &= \xi_{i} \circ \Upsilon(f).\\
    \end{align*}
    
    If $f \in \sP(k)^{+}$ then 
    $$ \Upsilon \circ \epsilon_k(f) = \Upsilon(f) $$ and 
    $$W_{k}\circ \Upsilon(f) = \Upsilon(f).$$ Thus we may assume that $f \in \sP(k+r)^{+}$ for some $r \geq 1$ in which case using Lemma \ref{continuity of specialization} we see
    \begin{align*}
        &\Upsilon \circ \epsilon_k(f) \\
        &=\Upsilon\left( \lim_{n}\epsilon_{k}^{(n)}(\Xi^{(n)}(f))\right) \\
        &= \Upsilon\left( \lim_{n}\frac{1}{[n-k]_t!} \sum_{\sigma \in \mathfrak{S}_{(1^k,n-k)}} t^{{n-k \choose 2} -\ell(\sigma)}T_{\sigma}\Xi^{(n)}(f) \right)\\
        &=  \lim_{n}\Upsilon \left(\frac{1}{[n-k]_t!} \sum_{\sigma \in \mathfrak{S}_{(1^k,n-k)}} t^{{n-k \choose 2} -\ell(\sigma)}T_{\sigma}\Xi^{(n)}(f) \right) \\
        &= \lim_{n} \sum_{\sigma \in \mathfrak{S}_{(1^k,n-k)}} \mathbbm{1}\left({n-k \choose 2} = \ell(\sigma)\right)\Upsilon \left(T_{\sigma}\Xi^{(n)}(f) \right)\\
        &= \lim_{n} (T_{n-1}\cdots T_{k+1})\cdots (T_{n-1}\cdots T_{k+r})\Upsilon(\Xi^{(n)}(f))\\
        &= \lim_{n} (\xi_{n-1}\cdots \xi_{k+1})\cdots (\xi_{n-1}\cdots \xi_{k+r})f(x_1,\ldots, x_n,0,\ldots ;0,0)\\
        &= \lim_{n} W_{k}^{(n)}f(x_1,\ldots, x_n,0,\ldots ;0,0)\\
        &= W_{k}\circ\Upsilon(f)\\
    \end{align*}

\end{proof}

\section{Almost Symmetric Schur Functions}

The stable-limit non-symmetric Macdonald functions $\widetilde{E}_{(\mu|\lambda)}$ are seen, from Proposition \ref{macd recursion}, to be generated by applying successive partial-symmetrization operators to the functions $\widetilde{E}_{(\mu*\lambda| \emptyset)}.$ Given that the operators $T_i,\epsilon_k$ specialize to the $\xi_i,W_k$ respectively, we may define a set of almost symmetric functions $s_{(\mu|\lambda)}$ analogously.

\begin{defn}\label{almost sym schur def}
    Define the \textbf{\textit{almost symmetric Schur functions}}, $s_{(\mu|\lambda)}= s_{(\mu|\lambda)}(x_1,x_2,\ldots)$, for $(\mu|\lambda) \in \Sigma$ by the following recursive formula:
    \begin{itemize}
        \item $s_{(\mu|\emptyset)} = \mathcal{K}_{\mu}$  
        \item If $\mu_r \geq \lambda_1$ then 
        $$s_{(\mu_1,\ldots, \mu_{r-1}| \mu_r, \lambda_1,\ldots, \lambda_{\ell})} = W_{r-1}(s_{(\mu_1,\ldots, \mu_{r-1}, \mu_r| \lambda_1,\ldots, \lambda_{\ell})}).$$
    \end{itemize}
\end{defn}

We will now compute a few non-trivial examples of almost symmetric Schur functions $s_{(\mu|\lambda)}.$

\begin{example}
    Here we calculate $s_{(2|3,1)}$ directly using the operators $\xi_i$ and $W_k$:
    \begin{align*}
        &s_{(2|3,1)}\\
        &= W_{1}W_{2}(s_{(2,3,1|\emptyset)})\\
        &= W_{1}W_{2} \xi_1 (s_{(3,2,1|\emptyset)})\\
        &= W_{1}W_{2}\xi_1 (x_1^3x_2^2x_3)\\
        &= W_{1}W_{2} ( x_1^3x_2^2x_3 + x_1^2x_2^3x_3)\\
        &= W_{1}(x_1^3x_2^2s_{1}[\mathfrak{X}_2] + x_1^2x_2^3s_{1}[\mathfrak{X}_2])\\
        &= x_1^{3} s_{(2,1)}[\mathfrak{X}_1] + x_1^{2} s_{(3,1)}[\mathfrak{X}_1]\\
    \end{align*}
\end{example}

\begin{example}
    Here we give a list of some examples of almost symmetric Schur functions that are neither symmetric Schur functions nor key polynomials.
    \begin{itemize}
        \item $s_{(0,1|2)} = x_1^2x_2 + x_1^2s_{1}[\mathfrak{X}_2] + x_2^2x_1 + x_2^2 s_{1}[\mathfrak{X}_2] + x_1s_2[\mathfrak{X}_2] + x_2s_2[\mathfrak{X}_2] + 2x_1x_2s_1[\mathfrak{X}_2]$
        \item $s_{(2|3,1)} = x_1^{3} s_{(2,1)}[\mathfrak{X}_1] + x_1^{2} s_{(3,1)}[\mathfrak{X}_1]$
        \item $s_{(2,1|1)} = x_1^2x_2 s_1[\mathfrak{X}_2]$
        \item $s_{(1,2|1)} = x_1^{2}x_2s_{1}[\mathfrak{X}_2] + x_1x_2^2s_{1}[\mathfrak{X}_2]$
        \item $s_{(1|2,1)} = x_1^{2}s_{(1,1)}[\mathfrak{X}_1] + x_1 s_{(2,1)}[\mathfrak{X}_1].$
    \end{itemize}
\end{example}

\begin{remark}
We note that from the above recursion it follows that for any $\lambda \in \Par$,
$s_{(\emptyset|\lambda)} = s_{\lambda}.$ Thus the almost symmetric Schur functions interpolate between the key polynomials and the Schur functions in infinitely many variables $x_1,x_2, \ldots.$ Lapointe in \cite{lapointe2022msymmetric} defines the \textbf{\textit{m-symmetric Schur functions}} $s_{(a;\lambda)}(x;t).$ These functions have the property that $s_{(a;\emptyset)}(x;t) = H_{a}(x;t)$ (the non-symmetric Hall-Littlewood polynomial) and $s_{(\emptyset;\lambda)}(x;t) = s_{\lambda}(x)$ similarly to the functions $s_{(\mu|\lambda)}(x)$ defined above. Further, they give a basis for $\sP(m)^{+}.$ However, it is not clear to this author how Lapointe's m-symmetric Schur functions are related to the almost symmetric Schur functions in this paper. Any proof that relates these two types of functions would likely be nontrivial and combinatorial in nature.
\end{remark}

We are now ready to compute the specializations of the stable-limit non-symmetric Macdonald functions $\widetilde{E}_{(\mu|\lambda)}$ at $q = t = 0.$

\begin{thm}\label{specialization theorem}
    For $(\mu|\lambda) \in \Sigma$, $\widetilde{E}_{(\mu|\lambda)} \in \mathcal{O}$ and
    $$\Upsilon(\widetilde{E}_{(\mu|\lambda)}) = s_{(\mu|\lambda)}(x).$$
\end{thm}
\begin{proof}
    Let $(\mu|\lambda) \in \Sigma.$ In order to show that $\widetilde{E}_{(\mu|\lambda)} \in \mathcal{O}$ it suffices by induction to verify that each $\epsilon_{k}(f) \in \mathcal{O}$ for every $f \in \mathcal{O}.$ However, this is easy to see using the explicit formula for the action of $\epsilon_{k}$ using the Jing vertex operators $\mathscr{B}_{r}$ (see Definition \ref{jing vertex op} and Lemma \ref{idempotent lemma}). We now proceed by direct computation using Lemma \ref{specializations of Hecke maps} and Proposition \ref{special case of specialization thm}.
    \begin{align*}
        &\Upsilon(\widetilde{E}_{(\mu|\lambda)})\\
        &= \Upsilon(\epsilon_{\ell(\mu)}(\widetilde{E}_{(\mu*\lambda|\emptyset)}))\\
        &= W_{\ell(\mu)}(\Upsilon(\widetilde{E}_{(\mu*\lambda|\emptyset)}))\\
        &= W_{\ell(\mu)}(\mathcal{K}_{\mu*\lambda})\\
        &= W_{\ell(\mu)}(s_{(\mu*\lambda|\emptyset}))\\
        &= s_{(\mu|\lambda)}.\\
    \end{align*}
\end{proof}

\subsection{Combinatorial Formula for Almost Symmetric Schur Functions}

In this section we will compute an explicit combinatorial formula for the monomial expansion of the almost symmetric Schur functions. Further, we will use this expansion to show that a generalization of the classical Kostka coefficients for Schur functions are non-negative integers. 

\begin{prop}\label{almost sym schur from key}
    For $(\mu|\lambda) \in \Sigma$,
    $$s_{(\mu|\lambda)} = \lim_{n} \mathcal{K}_{\mu*0^{n}*\rev(\lambda)}.$$
\end{prop}
\begin{proof}
We proceed by direct calculation:
    \begin{align*}
        &s_{(\mu|\lambda)}\\
        &=W_{\ell(\mu) }\cdots W_{\ell(\mu)+\ell(\lambda)} s_{(\mu*\lambda|\emptyset)}\\
        &=W_{\ell(\mu) } s_{(\mu*\lambda|\emptyset)}\\
        &= W_{\ell(\mu)} \mathcal{K}_{\mu*\lambda} \\
        &= \lim_{n} W_{\ell(\mu)}^{(\ell(\mu) + \ell(\lambda) + n)}(\mathcal{K}_{\mu*\lambda*0^{n}}) \\
        &= \lim_{n} (\xi_{\ell(\mu)+\ell(\lambda)+n-1}\cdots \xi_{\ell(\mu)+1})\cdots (\xi_{\ell(\mu)+\ell(\lambda)+ n-1}\cdots \xi_{\ell(\mu)+\ell(\lambda)})(\mathcal{K}_{\mu*\lambda*0^{n}}) \\
        &= \lim_{n}\mathcal{K}_{\mu*0^{n}*\rev(\lambda)}.\\
    \end{align*}
\end{proof}

As an immediate consequence we get the following:

\begin{cor}\label{almost sym schur are basis}
    The set $\{s_{(\mu|\lambda)}(x)|(\mu|\lambda) \in \Sigma\}$ is a homogeneous $\mathbb{Q}$-basis for $\sP_{as,\mathbb{Q}}^{+}.$
\end{cor}
\begin{proof}
    Since the key polynomials are homogeneous and the operators $W_k$ are clearly homogeneous, we see that the $s_{(\mu|\lambda)}$ are homogeneous as well. As there are sufficiently many $s_{(\mu|\lambda)}$ in each homogeneous component of $\sP(k)^{+}$, it suffices to show that the $s_{(\mu|\lambda)}$ are linearly independent (over $\mathbb{Q}$). Let $(\mu^{(1)}|\lambda^{(1)}),\ldots, (\mu^{(m)}|\lambda^{(m)}) \in \Sigma$ be distinct. Set $r^{(i)}:= \ell(\mu^{(i)})+\ell(\lambda^{(i)}).$ Suppose that for some $a^{(i)} \in \mathbb{Q}$, $\sum_{i=1}^{m} a^{(i)}s_{(\mu^{(i)}|\lambda^{(i)})} =0.$ Then 
    \begin{align*}
        0 &= \sum_{i=1}^{m} a^{(i)}s_{(\mu^{(i)}|\lambda^{(i)})} \\
        &= \sum_{i=1}^{m} a^{(i)} \lim_{n} \mathcal{K}_{\mu^{(i)}*0^{n-r^{(i)}}*\rev(\lambda^{(i)})} \\
        &=  \lim_{n} \sum_{i=1}^{m} a^{(i)} \mathcal{K}_{\mu^{(i)}*0^{n-r^{(i)}}*\rev(\lambda^{(i)})}.\\
    \end{align*}
    Now we see that for all sufficiently large $n$, 
    $$\sum_{i=1}^{m} a^{(i)} \mathcal{K}_{\mu^{(i)}*0^{n-r^{(i)}}*\rev(\lambda^{(i)})} = 0$$
    but, since the pairs $(\mu^{(i)}|\lambda^{(i)})$ are distinct, we know that the key polynomials $\mathcal{K}_{\mu^{(i)}*0^{n-r^{(i)}}*\rev(\lambda^{(i)})}$ are linearly independent. Therefore, $a^{(i)} = 0$ as desired.
\end{proof}

\begin{remark}
    It is an interesting question whether or not the $s_{(\mu|\lambda)}$, which as we will show in Theorem \ref{combinatorial formula for almost sym schur} have integral coefficients in the monomial basis, are a basis over $\mathbb{Z}$ for the space $\sP_{as, \mathbb{Z}}^{+}$ of almost symmetric functions over $\mathbb{Z}.$ This is non-trivial and does not follow from Corollary \ref{almost sym schur are basis}. However, this seems likely as there should be a simple monomial ordering with respect to which the $s_{(\mu|\lambda)}$ are uni-triangular. This ordering would need to specialize to both the Bruhat ordering on finite variable monomials and the dominance ordering on partitions in both the fully non-symmetric and symmetric extremes.
\end{remark}

In order to describe a combinatorial model for the almost symmetric Schur functions we require the next definition.

\begin{defn}
    Let $(\mu|\lambda) \in \Sigma.$ Let $\omega$ denote the first infinite ordinal i.e. $n < \omega$ for all $n \in \{1,2,\ldots\}.$ For a labelling $\sigma : dg'(\mu*\rev(\lambda)) \rightarrow \{1,2,\ldots\}$ denote by $\sigma^{\star}$ the labelling of $\widehat{dg}(\mu* \rev(\lambda))$ given by 
    \begin{itemize}
        \item $\sigma^{\star}(u) = \sigma(u)$ if $u \in dg'(\mu*\rev(\lambda))$
        \item $\sigma^{\star}(j,0) = j $ for $1 \leq j \leq \ell(\mu)$
        \item $\sigma^{\star}(j,0) = \omega + j - \ell(\mu) -1 $ for $\ell(\mu) +1 \leq j \leq \ell(\mu) + \ell(\lambda).$
    \end{itemize}
    We naturally extend the definitions in Definition \ref{HHL defn} of non-attacking, $\coinv$, and $\Des$ to labellings of the form $\sigma^{\star}$ which take values in $\{1,2,\ldots\} \cup \{\omega +1,\omega+2,\ldots \}.$ Define $\mathcal{L}(\mu|\lambda)$ to be the set of labellings $\sigma : dg'(\mu*\rev(\lambda)) \rightarrow \{1,2,\ldots\}$ such that $\sigma^{\star}$ is non-attacking, $\coinv(\sigma^{\star}) = 0$, and $\Des(\sigma^{\star}) = \emptyset.$
\end{defn}

\begin{example}
We will consider in this example two labellings of the type defined above for the pair $(2|3,1)$.
Our diagrams in this case are given as follows:
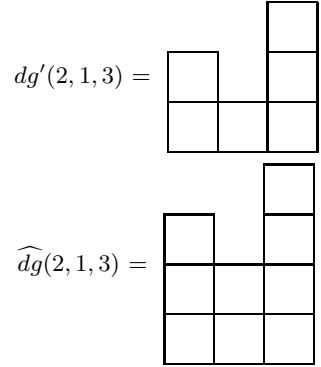
\begin{center}
$dg'(2,1,3)$ = 
\ytableausetup{centertableaux, boxframe= normal, boxsize= 2em}
\begin{ytableau}
 \none &   \none    &    \\
       &   \none    &    \\
       &            &    \\
\end{ytableau}

\end{center}

\begin{center}
$\widehat{dg}(2,1,3)$ =
\ytableausetup{centertableaux, boxframe= normal, boxsize= 2em}
\begin{ytableau}
 \none &   \none    &    \\
       &   \none    &    \\
       &            &    \\
       &            &    \\
\end{ytableau}

\end{center}

Consider the labellings $\sigma_1,\sigma_2:dg'(2,1,3) \rightarrow \{1,2,3,4\}$ and there corresponding labellings $\sigma_1^{\star},\sigma_2^{\star}:\widehat{dg}(2,1,3) \rightarrow \{1,2,3,4\}$ given by

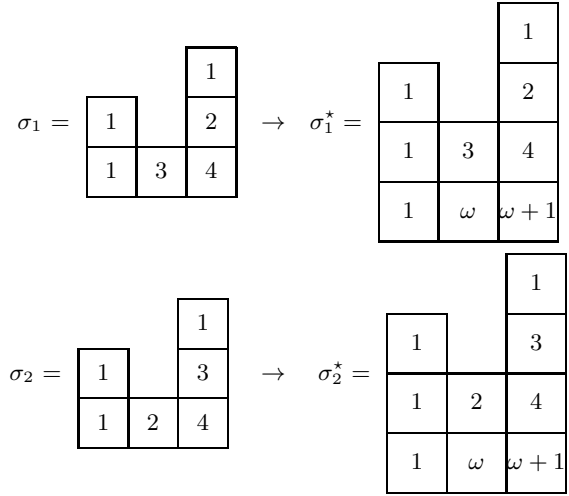
\begin{center}
$\sigma_1$ =
\ytableausetup{centertableaux, boxframe= normal, boxsize= 2em}
\begin{ytableau}
 \none &   \none    &  1  \\
   1    &   \none    &  2  \\
   1    &      3      &  4  \\
\end{ytableau}
~~$\rightarrow$ ~~$\sigma_1^{\star} =$
\ytableausetup{centertableaux, boxframe= normal, boxsize= 2.4em}
\begin{ytableau}
 \none &   \none    &  1  \\
   1    &   \none    &  2  \\
   1    &      3      &  4  \\
   1    &     \omega       &  \omega + 1  \\
\end{ytableau}

\end{center}

\begin{center}
$\sigma_2$ =
\ytableausetup{centertableaux, boxframe= normal, boxsize= 2em}
\begin{ytableau}
 \none &   \none    &  1  \\
   1    &   \none    &  3  \\
   1    &      2      &  4  \\
\end{ytableau}
~~ $\rightarrow$ ~~ $\sigma_2^{\star} =$
\ytableausetup{centertableaux, boxframe= normal, boxsize= 2.4em}
\begin{ytableau}
 \none &   \none    &  1  \\
   1    &   \none    &  3  \\
   1    &      2      &  4  \\
   1    &     \omega       &  \omega + 1  \\
\end{ytableau}

\end{center}

Both $\sigma_1,\sigma_2$ are non-attacking with $\maj(\sigma_1^{\star}) = \maj(\sigma_2^{\star}) = 0.$ However, $\coinv(\sigma_1^{\star}) = 0$ whereas $\coinv(\sigma_2^{\star}) \neq 0.$ To see this note that in the labelling $\sigma_2,$
the boxes 
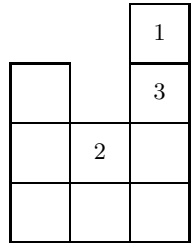
\begin{center}
\begin{ytableau}
 \none &   \none    &  1  \\
       &   \none    &  3  \\
       &      2      &    \\
       &            &    \\
\end{ytableau}
\end{center}

form a co-inversion triple of Type 2.
\end{example}

We may now give a combinatorial formula for the almost symmetric Schur functions. This formula is derived directly from the HHL-type formula for the key polynomials.

\begin{thm}\label{combinatorial formula for almost sym schur}

    For $(\mu|\lambda)\in \Sigma$
    $$s_{(\mu|\lambda)} = \sum_{\sigma \in \mathcal{L}(\mu|\lambda)} x^{\sigma}.$$
\end{thm}
\begin{proof}
    We start by noticing that from Proposition \ref{almost sym schur from key} we have 
    \begin{align*}
        &s_{(\mu|\lambda)} \\
        &= \lim_{n} \mathcal{K}_{\mu*0^n*\rev(\lambda)}\\
        &= \lim_{n} \sum_{\sigma \in \mathcal{L}(\mu*0^n*\rev(\lambda))} x^{\sigma}.\\
    \end{align*}

    For all $n \geq 0$ there is an injection $ \mathcal{L}(\mu*0^n*\rev(\lambda)) \rightarrow \mathcal{L}(\mu*0^{n+1}*\rev(\lambda))$ obtained as follows. Let $\sigma \in \mathcal{L}(\mu*0^n*\rev(\lambda)).$ Consider $\sigma': dg'(\mu*0^{n+1}*\rev(\lambda)) \rightarrow [n+1+\ell(\mu)+\ell(\lambda)]$ given by 
    \begin{itemize}
        \item $\sigma'(u) = \sigma(u)$ if $u \in dg'(\mu)$
        \item $\sigma'(i,j) = \sigma(i,j-1)$ if $(i,j)$ lies in the $\rev(\lambda)$ component of $dg'(\mu*0^{n+1}*\rev(\lambda)).$
    \end{itemize}
    In other words, we are simply aligning the $\rev(\lambda)$ parts of each of the diagrams $dg'(\mu*0^{n+1}*\rev(\lambda))$ and $dg'(\mu*0^{n}*\rev(\lambda))$ and copying the corresponding values of $\sigma.$ It is easy to see that $\sigma' \in \mathcal{L}(\mu*0^{n+1}*\rev(\lambda))$ and that the map $\sigma \rightarrow \sigma'$ is injective. To see this note that the entries values of $\sigma'$ are weakly decrease upwards along columns so that $\maj(\widehat{\sigma'}) = 0$ and, since $\widehat{\sigma}$ has no co-inversion triples of Types 1 or 2, then neither does $\widehat{\sigma'}$ meaning that $\coinv(\widehat{\sigma'}) = 0.$ Now we may consider the directed union 
    $$L:= \bigcup_{n \geq 0} \mathcal{L}(\mu*0^{n}*\rev(\lambda))$$ where we identify the image of $\mathcal{L}(\mu*0^{n}*\rev(\lambda))$ in $\mathcal{L}(\mu*0^{n+1}*\rev(\lambda))$ for all $n \geq 0.$ Hence, we have 
    $$s_{(\mu|\lambda)} = \sum_{\sigma \in L} x^{\sigma}.$$ 
    
    Lastly, we show that there exists a simple bijection $L \rightarrow \mathcal{L}(\mu|\lambda)$ such that $x^{\sigma} = x^{f(\sigma)}$ for all $\sigma \in L.$ For $\sigma \in L$ say, $\sigma \in \mathcal{L}(\mu*0^{n}*\rev(\lambda))$, we may define $\sigma'': dg'(\mu*\rev(\lambda)) \rightarrow \{1, 2,\ldots\}$ by 
    \begin{itemize}
        \item $\sigma''(u) = \sigma(u)$ if $u \in dg'(\mu)$
        \item $\sigma''(i,j) = \sigma(i+n,j)$ for $(i,j)$ in the $\rev(\lambda)$ component of $dg'(\mu*\rev(\lambda)).$
    \end{itemize}

    Then $\sigma'' \in \mathcal{L}(\mu|\lambda)$ and the map $\sigma \rightarrow \sigma''$ is injective. We now show this map is also surjective. Let $\gamma \in \mathcal{L}(\mu|\lambda)$ and $N:= \max\{\max_{u \in dg'(\mu*\rev(\lambda))}{\sigma(u)}, \ell(\mu)+\ell(\lambda)\}.$ Define $\sigma: \mu*0^{N - \ell(\mu)-\ell(\lambda)}*\rev(\lambda) \rightarrow [N]$ similarly to before by copying the values of $\sigma$ for both the $\mu$ and $\rev(\lambda)$ components of $dg'(\mu*\rev(\lambda))$ onto the corresponding components of $dg'(\mu*0^{N - \ell(\mu)-\ell(\lambda)}*\rev(\lambda)).$ Since $N$ was chosen sufficiently large, $\sigma \in \mathcal{L}(\mu*0^{N - \ell(\mu)-\ell(\lambda)}*\rev(\lambda)).$ Now $\sigma'' =\gamma$ and $x^{\sigma''} = x^{\gamma}.$ Therefore, 
    $$s_{(\mu|\lambda)} = \sum_{\sigma \in L} x^{\sigma} = \sum_{\sigma \in \mathcal{L}(\mu|\lambda)} x^{\sigma}.$$
\end{proof}

\subsection{Almost Symmetric Kostka Coefficients}

Since $s_{(\mu|\lambda)} \in \sP(\ell(\mu))^{+}$ and the set $\{x^{\alpha}m_{\nu}[\mathfrak{X}_{\ell(\mu)}]~| (\alpha|\nu) \in \Sigma, \ell(\alpha) \leq \ell(\mu) \}$ is a basis for $\sP(k)^{+}$ we may consider the following definition.

\begin{defn}\label{almost symmetric Kostka coefficients defn}
    Define the \textbf{\textit{almost symmetric Kostka coefficients}} $K_{(\alpha|\nu)}^{(\mu|\lambda)}$ to be the coefficients of the almost symmetric Schur functions expanded into the monomial basis of $\sP(\ell(\mu))^{+}$, i.e.
    $$s_{(\mu|\lambda)} = \sum_{\substack{(\alpha|\nu) \\ \ell(\alpha) \leq \ell(\mu)}} K_{(\alpha|\nu)}^{(\mu|\lambda)} x^{\alpha}m_{\nu}[\mathfrak{X}_{\ell(\mu)}].$$ 
    If $\ell(\alpha) > \ell(\mu)$ we simply set $K_{(\alpha|\nu)}^{(\mu|\lambda)} = 0.$
\end{defn}

\begin{remark}
    It is straightforward to check that  
    $$K_{(\mu|\nu)}^{(\emptyset|\lambda)} = \delta_{\mu,\emptyset} K_{\lambda, \nu}$$
    meaning that the $K_{(\alpha|\nu)}^{(\mu|\lambda)}$ generalize the classical Kostka coefficients $K_{\lambda, \nu}.$ On the other extreme, we find that 
    $$K_{(\alpha|\lambda)}^{(\mu|\emptyset)} = 0$$
    unless $\lambda = \emptyset$ in which case $K_{(\alpha|\emptyset)}^{(\mu|\emptyset)}$ is the multiplicity of the weight $\alpha$ in the Demazure character corresponding to $\mu.$ In either case, we see that the Kostka coefficients are non-negative.
\end{remark}

Using the combinatorial formula we found for the $s_{(\mu|\lambda)}$ (Theorem \ref{combinatorial formula for almost sym schur}) we are able to give a simple proof that the almost symmetric Kostka coefficients are non-negative integers.

\begin{thm}[Positivity for almost symmetric Kostka coefficients]\label{Positivity for almost symmetric Kostka coefficients}
    $$K_{(\alpha|\nu)}^{(\mu|\lambda)} \in \mathbb{Z}_{\geq 0}$$
\end{thm}
\begin{proof}
    Let $(\mu|\lambda) \in \Sigma.$ Using the explicit combinatorial formula in Theorem \ref{combinatorial formula for almost sym schur} we see that 
    $$s_{(\mu|\lambda)} = \sum_{\sigma \in \mathcal{L}(\mu|\lambda)} x^{\sigma}.$$ However, we know $s_{(\mu|\lambda)}$ is symmetric in the variables $x_{\ell(\mu)+1}, x_{\ell(\mu)+2},\ldots$ so we may group terms by symmetry to find 
    $$\sum_{\sigma \in \mathcal{L}(\mu|\lambda)} x^{\sigma} = \sum_{\nu \in \Par} m_{\nu}[\mathfrak{X}_{\ell(\mu)}] \sum_{\sigma \in L_{\nu}(\mu|\lambda)} x_1^{|\sigma^{-1}(1)|}\cdots x_{\ell(\mu)}^{|\sigma^{-1}(\ell(\mu))|} $$
    where $L_{\nu}(\mu|\lambda)$ is the set of labellings $\sigma: \mu*\rev(\lambda) \rightarrow [\mu+\ell(\nu)]$ such that $\sigma \in \mathcal{L}(\mu|\lambda)$ and for all $1 \leq i \leq \ell(\nu)$, $|\sigma^{-1}(\ell(\mu)+i)| = \nu_i.$ Notice that $|L_{\nu}(\mu|\lambda)| < \infty$ for all $\nu.$ 
    
    We may further subdivide the sets $L_{\nu}(\mu|\lambda)$ now to account for the value of $x_1^{|\sigma^{-1}(1)|}\cdots x_{\ell(\mu)}^{|\sigma^{-1}(\ell(\mu))|}.$ For $\ell(\alpha) \leq \ell(\mu)$ let $L_{(\alpha|\nu)}(\mu|\lambda)$ denote the set of all $\sigma \in L_{\nu}(\mu|\lambda)$ such that $|\sigma^{-1}(i)| = (\alpha*0^{\ell(\mu)-\ell(\alpha)})_i.$ Then 
    $$s_{(\mu|\lambda)} = \sum_{(\alpha|\nu)} |L_{(\alpha|\nu)}(\mu|\lambda)| x^{\alpha} m_{\nu}[\mathfrak{X}_{\ell(\mu)}].$$

    Thus 
    $$K_{(\alpha|\nu)}^{(\mu|\lambda)} = |L_{(\alpha|\nu)}(\mu|\lambda)| 
 \in \mathbb{Z}_{\geq 0}.$$
\end{proof}

\begin{remark}
    Note that $K_{(\alpha|\nu)}^{(\mu|\lambda)} = |L_{(\alpha|\nu)}(\mu|\lambda)|$ gives a combinatorial formula for the almost symmetric Kostka coefficients. This formula generalizes the well known formula $K_{\lambda,\mu} = |\SSYT(\lambda,\mu)|$ where $\SSYT(\lambda,\mu)$ is the set of semistandard Young tableaux with shape $\lambda$ and content $\mu$.
\end{remark}

\begin{example}
    We saw before that 
    $$s_{(2|3,1)} = x_1^{3}s_{(2,1)}[\mathfrak{X}_1] + x_1^2 s_{(3,1)}[\mathfrak{X}_1]$$
    which we can expand as 
    $$s_{(2|3,1)} = x_1^3 m_{(2,1)}[\mathfrak{X}_1] + 2x_1^3m_{(1,1,1)}[\mathfrak{X}_1] + x_1^2m_{(3,1)}[\mathfrak{X}_1] + x_1^2m_{(2,2)}[\mathfrak{X}_1] + 2x_1^2m_{(2,1,1)}[\mathfrak{X}_1] + 3x_1^2m_{(1,1,1,1)}[\mathfrak{X}_1].$$
    This gives that, for example, 
    $K_{(2|1,1,1,1)}^{(2|3,1)} = 3$
    which corresponds to the 3 diagrams: 
    \begin{center}
        \begin{ytableau}
 \none &   \none    &  2  \\
   1    &   \none    &  3  \\
   1    &      4      &  5  \\
   1    &     \omega       &  \omega + 1  \\
\end{ytableau}
~
 \begin{ytableau}
 \none &   \none    &  2  \\
   1    &   \none    &  4  \\
   1    &      3      &  5  \\
   1    &     \omega       &  \omega + 1  \\
\end{ytableau} 
~
 \begin{ytableau}
 \none &   \none    &  3  \\
   1    &   \none    &  4  \\
   1    &      2      &  5  \\
   1    &     \omega       &  \omega + 1  \\
\end{ytableau}
    \end{center}

    Note that the above fillings in the $\rev(\lambda) = (1,3)$ component are exactly, up to shifting indices, the semistandard Young tableuax of shape $(3,1)$ with content $(1,1,1,1)$ (and hence standard). This reflects that in the monomial-Schur expansion of $s_{(2|3,1)}$ there is one copy of $x_1^2s_{(3,1)}[\mathfrak{X}_1]$ and $K_{(3,1),(1,1,1,1)} = 3.$

    We also have that $K^{(2|3,1)}_{(3|1,1,1,1)} = 2$ which may be seen by computing the labellings in $L_{(3|1,1,1)}(2|3,1)$ directly:

\begin{center}
        \begin{ytableau}
 \none &   \none    &  2  \\
   1    &   \none    &  3  \\
   1    &      1      &  4  \\
   1    &     \omega       &  \omega + 1  \\
\end{ytableau}
~
 \begin{ytableau}
 \none &   \none    &  1  \\
   1    &   \none    &  2  \\
   1    &      3      &  4  \\
   1    &     \omega       &  \omega + 1  \\
\end{ytableau} 

    \end{center}

From the computation of $K_{(2|1,1,1,1)}^{(2|3,1)} $ one might be tempted to guess that there is always a way to compute the almost symmetric Kostka numbers by classical Kostka numbers in some obvious manner. However, the example of $K^{(2|3,1)}_{(3|1,1,1,1)}$ shows that it is not always so simple. It particular, the filling 
\begin{center}
        \begin{ytableau}
 \none &   \none    &  1  \\
   1    &   \none    &  3  \\
   1    &      2      &  4  \\
   1    &     \omega       &  \omega + 1  \\
\end{ytableau}
    \end{center}
    has a reverse standard filling of $\rev(\lambda)$ but is not in $L_{(3|1,1,1)}(2|3,1)$ since $\coinv \neq 0.$
\end{example}

\section{Parabolic Demazure Character Formula}

In this section we are going to show that the monomial-Schur expansion of $s_{(\mu|\lambda)}$ has non-negative coefficients using the Demazure character formula by relating $s_{(\mu|\lambda)}$ to the representation theory of parabolic subgroups of type $\GL.$ This is a strictly stronger result than Theorem \ref{Positivity for almost symmetric Kostka coefficients} and we do not find a simple combinatorial formula for the coefficients of this expansion. The main result of this section, Theorem \ref{rep theory for schur expansion}, may be viewed as a Demazure formula for parabolic subgroups of $\GL$ although it will follow almost directly from the usual Demazure formula. 

\begin{defn}
   Define the scalars $M_{(\alpha|\nu)}^{(\mu|\lambda)}$ to be the coefficients of the expansion of the almost symmetric Schur functions into the monomial-Schur basis of $\sP(\ell(\mu))^{+}$, i.e.
    $$s_{(\mu|\lambda)} = \sum_{\substack{(\alpha|\nu) \\ \ell(\alpha) \leq \ell(\mu)}} M_{(\alpha|\nu)}^{(\mu|\lambda)} x^{\alpha}s_{\nu}[\mathfrak{X}_{\ell(\mu)}].$$
    If $\ell(\alpha) > \ell(\mu)$ we simply set $M_{(\alpha|\nu)}^{(\mu|\lambda)} = 0.$
\end{defn}

We wish to show that $M_{(\alpha|\nu)}^{(\mu|\lambda)} \in \mathbb{Z}_{\geq 0}$ but in order to do so we must first review some representation theory in type $\GL.$

\begin{defn}
    Let $n \geq 1.$ Define $\bor_n$ to be the Borel subgroup of upper-triangular matrices in $\GL_n$ and let $\tor_n$ denote the group of diagonal matrices in $\GL_n.$ For $0 \leq k \leq n$ denote by $\para_n(k)$ the group of $M \in \GL_n$ such that $M_{ij}=0$ if either $1 \leq j < i \leq k$ or $j \leq k \leq i-1.$ Lastly, let $\levi_{n}(k) = \tor_k \times \GL_{n-k} \subset \GL_n$ under the block diagonal embedding $\GL_k \times \GL_{n-k} \rightarrow \GL_n.$ Let $\mathfrak{b}_n$ denote the Lie algebra of $\bor_n$ i.e. the set of upper triangular $n\times n$ matrices over $\mathbb{C}$ with the usual commutator. Let $\mathcal{U}(\mathfrak{b}_n)$ denote the \textbf{\textit{universal enveloping algebra}} of $\mathfrak{b}_n$. For a dominant integral weight $\lambda \in \mathbb{Z}_{\geq 0}^n$ let $\mathcal{V}^{\lambda}$ denote the corresponding highest weight representation of $\GL_n.$
\end{defn}

\begin{defn}
    Given a finite dimensional polynomial representation $V$ of $\tor_n$ we will denote by $\Char(V) \in \mathbb{Z}[x_1,\ldots,x_n]$ the \textbf{\textit{ formal character}} of $V$ as 
    $$\Char(V) = \sum_{\alpha \in \mathbb{Z}_{\geq 0}^{n}} \Dim \Hom_{\tor_n}(\alpha,V) x^{\alpha}.$$
\end{defn}

\begin{defn}\cite{Dem_1974}
    Given a dominant integral weight $\lambda \in \mathbb{Z}_{\geq 0}^{n}$ and $\sigma \in \mathfrak{S}_n$ define the \textbf{\textit{Demazure module}} $\mathcal{V}^{\lambda}_{\sigma(\lambda)}$ to be the $\bor_n$-module $$\mathcal{V}^{\lambda}_{\sigma}:= \mathcal{U}(\mathfrak{b}_n)v$$ where
    $v \in \mathcal{V}^{\lambda}_{\sigma}$ is any weight vector with weight $\sigma(\lambda).$ 
\end{defn}

\begin{remark}
    Notice that the Demazure module $\mathcal{V}^{\lambda}_{\sigma}$ is only well defined up to the vector $\sigma(\lambda).$ Therefore, we may instead index these modules as 
$$\mathcal{V}^{\lambda}_{\sigma(\lambda)}:= \mathcal{V}^{\lambda}_{\sigma}.$$
\end{remark}

\begin{thm}(Demazure Character Formula)\cite{Andersen1985}
    Given a dominant integral weight $\lambda$ and $\sigma \in \mathfrak{S}_n$
    $$\Char(\mathcal{V}^{\lambda}_{\sigma(\lambda)}) = \mathcal{K}_{\sigma(\lambda)}.$$
\end{thm}

\begin{remark}
    The Demazure character formula was first conjectured by Demazure \cite{Dem_1974} but 
    the first complete proof was given by Andersen \cite{Andersen1985} by realizing the Demazure modules as spaces of sections of vector bundles of Schubert varieties and showing that the singularities of Schubert varieties are rational.
\end{remark}

\begin{defn}
    Let $(\mu|\lambda) \in \Sigma.$ For all $n \geq \ell(\mu)+\ell(\lambda)$ define 
    $$\mathcal{V}^{(n)}(\mu|\lambda):= \mathcal{V}^{\sort(\mu*\lambda)*0^{n-\ell(\sort(\mu*\lambda))}}_{\mu*0^{n-\ell(\mu)-\ell(\lambda)}*\rev(\lambda)}.$$ If $\alpha \in \Compred$ and $\ell(\alpha) \leq k$ we will write $\chi^{(n)}(\alpha|\lambda)$ for the irreducible $\levi_{n}(k) = \tor_{k} \times \GL_{n-k}$- module given by
    $$\chi^{(n)}(\alpha|\lambda):= (\alpha*0^{k-\ell(\alpha)}) \otimes \mathcal{V}^{\lambda*0^{n-k- \ell(\lambda)}}$$ where we are using the shorthand $\alpha*0^{k-\ell(\alpha)}$ to represent the corresponding $1$-dimensional representation of $\tor_{k}.$ 
\end{defn}

We may relate the almost symmetric Schur functions  $s_{(\mu|\lambda)}$ to Demazure characters via key polynomials directly from the following simple lemma.

\begin{lem}\label{rep theory of almost sym schur} 
Let $(\mu|\lambda) \in \Sigma$. Then
    $$s_{(\mu|\lambda)} = \lim_{n} \Char\mathcal{V}^{(n)}(\mu|\lambda).$$
\end{lem}
\begin{proof}
    In Proposition \ref{almost sym schur from key} we saw that 
    $$s_{(\mu|\lambda)}= \lim_{n} \mathcal{K}_{\mu*0^n*\rev(\lambda)} = \lim_{n} \mathcal{K}_{\mu*0^{n-\ell(\mu)-\ell(\lambda)}*\rev(\lambda)}.$$
    Using the Demazure character formula we see that 
    $$\mathcal{K}_{\mu*0^{n-\ell(\mu)-\ell(\lambda)}*\rev(\lambda)} = \Char\left( \mathcal{V}^{\sort(\mu*\lambda)*0^{n-\ell(\sort(\mu*\lambda))}}_{\mu*0^{n-\ell(\mu)-\ell(\lambda)}*\rev(\lambda)}\right)$$ so the result follows.
\end{proof}

\subsection{Positivity of Monomial-Schur Expansion Coefficients}

We require the following simple lemma.

\begin{lem}\label{levi submodule lemma}
    Suppose $\lambda$ is an integral dominant weight of $\GL_n$ and $ \alpha* \beta = \sigma(\lambda)$ for some $\sigma \in \mathfrak{S}_n$ with $\beta$ weakly decreasing. Then $\mathcal{V}^{\lambda}_{\alpha*\beta}$ is a $\para_n(\ell(\alpha))$ submodule of $\mathcal{V}^{\lambda}.$
\end{lem}
\begin{proof}
    Let $k = \ell(\alpha).$
    Since $\para_n(k)$ is the semidirect product of $\bor_n$ and $\levi_n(k)$ we only need to show that $\mathcal{V}^{\lambda}_{\alpha*\beta}$ is preserved under the action of both $\bor_n$ and $\levi_n(k)$. Since $\mathcal{V}^{\lambda}_{\alpha*\beta}$ is by definition a $\bor_n$-module it suffices to show that $\mathcal{V}^{\lambda}_{\alpha*\beta}$ is preserved under the action of $Id_{k}\times \GL_{n-k}.$ 

    We will proceed by induction. To start fix $v_0 \in \mathcal{V}^{\lambda}_{\alpha*\beta}$ to be a nonzero vector with weight $\alpha*\beta.$ Then for all $k+1 \leq i< j \leq n$, since $\beta$ is weakly decreasing, $E_{ji}v = 0 \in \mathcal{V}^{\lambda}_{\alpha*\beta}.$ Suppose now that $v_0,v_1,\ldots , v_{m+1}$ is a sequence of weight vectors in $\mathcal{V}^{\lambda}_{\alpha*\beta}$ with $v_{r+1} = E_{i_rj_r}v_r$ for all $0 \leq r \leq m$ for some $1 \leq i_r < j_r \leq n$ and that 
    $$E_{ji}v_r \in  \mathcal{V}^{\lambda}_{\alpha*\beta}$$ for all $k+1 \leq i< j \leq n$ and $ 0 \leq r \leq m.$ Note that any weight vector in $\mathcal{V}^{\lambda}_{\alpha*\beta}$ may be obtained using such a chain. Now fix some $k+1 \leq i< j \leq n.$ We see that 
    \begin{align*}
        &E_{ji}v_{m+1}\\
        &= E_{ji}E_{i_mj_m}v_m\\
        &= \left(E_{i_mj_m}E_{ji} + [E_{ji},E_{i_mj_m}] \right) v_m\\
        &= E_{i_mj_m}\left( E_{ji} v_m\right) + [E_{ji},E_{i_mj_m}]v_m.\\
    \end{align*}
    By assumption $E_{ji}v_m \in \mathcal{V}^{\lambda}_{\alpha*\beta}$ so that, since $i_m < j_m$, 
    $E_{i_mj_m}\left( E_{ji} v_m\right) \in \mathcal{V}^{\lambda}_{\alpha*\beta}.$ Therefore, it suffices to show that $[E_{ji},E_{i_mj_m}]v_m \in \mathcal{V}^{\lambda}_{\alpha*\beta}.$
    
    There are a few cases we must consider. First, assume $i = i_m.$ Then 
    $$[E_{ji},E_{i_mj_m}]v_m = \left(E_{jj_m} - \delta_{j,j_m}E_{ii} \right)v_m = E_{jj_m}v_m - cv_m$$ for some scalar $c$. If $j \leq j_m$ then $E_{jj_m}v_m \in \mathcal{V}^{\lambda}_{\alpha*\beta}$ automatically. If instead $j > j_m$, then $k+1 \leq i = i_m < j_m$ so 
    $E_{jj_m}v_m \in \mathcal{V}^{\lambda}_{\alpha*\beta}$ by the inductive hypothesis. Either way $[E_{ji},E_{i_mj_m}]v_m \in \mathcal{V}^{\lambda}_{\alpha*\beta}.$ 
    
    Now assume $j = j_m.$ Then 
    $$[E_{ji},E_{i_mj_m}]v_m = \left( \delta_{i,i_m}E_{jj} - E_{i_mi} \right)v_m = cv_m - E_{i_mi}v_m$$ for some scalar $c.$ If $i_m \leq i$ then $E_{i_mi}v_m \in \mathcal{V}^{\lambda}_{\alpha*\beta}$ automatically. If $i_m > i$ then, since $k+1 \leq i$, $E_{i_mi}v_m \in \mathcal{V}^{\lambda}_{\alpha*\beta}$ by the inductive hypothesis. In either case, $[E_{ji},E_{i_mj_m}]v_m \in \mathcal{V}^{\lambda}_{\alpha*\beta}.$ Lastly, if $i\neq i_m$ and $j \neq j_m$ then $[E_{ji},E_{i_mj_m}] = 0$ so $[E_{ji},E_{i_mj_m}]v_m = 0 \in \mathcal{V}^{\lambda}_{\alpha*\beta}$ trivially.
\end{proof}

Since the group $\levi_{n}(k)$ is reductive we obtain the following representation theoretic interpretation for the coefficients $M_{(\alpha|\gamma)}^{(\mu|\lambda)}.$

\begin{thm}\label{rep theory for schur expansion}
Let $(\mu|\lambda),(\alpha|\gamma) \in \Sigma.$ For all sufficiently large n
$$M_{(\alpha|\gamma)}^{(\mu|\lambda)} = \Dim \Hom_{\levi_{n}(\ell(\mu))}\left( \chi^{(n)}(\alpha|\nu),\mathcal{V}^{(n)}(\mu|\lambda) \right) \in \mathbb{Z}_{\geq 0}.$$
\end{thm}
\begin{proof}
    From Lemma \ref{rep theory of almost sym schur} and the definition of the coefficients $M_{(\alpha|\gamma)}^{(\mu|\lambda)}$ we see that for $n$ sufficiently large 
    $$ \sum_{\substack{(\alpha|\nu) \\ \ell(\alpha) \leq \ell(\mu)}} M_{(\alpha|\nu)}^{(\mu|\lambda)} x^{\alpha}s_{\nu}[x_{\ell(\mu)+1}+\ldots + x_{n}] = \Char\mathcal{V}^{(n)}(\mu|\lambda).$$ From Lemma \ref{levi submodule lemma} we may decompose $\mathcal{V}^{(n)}(\mu|\lambda)$ into irreducible $\levi_{n}(\ell(\mu))$ submodules as
    $$\mathcal{V}^{(n)}(\mu|\lambda) = \bigoplus_{\substack{(\alpha|\nu) \\ \ell(\alpha) \leq \ell(\mu)}} \chi^{(n)}(\alpha|\nu)^{\oplus d^{(n)}_{(\alpha|\nu)}}$$ where 
    $d^{(n)}_{(\alpha|\nu)} = \Dim \Hom_{\levi_{n}(\ell(\mu))}\left( \chi^{(n)}(\alpha|\nu),\mathcal{V}^{(n)}(\mu|\lambda) \right).$
    Notice that 
    $$\Char\chi^{(n)}(\alpha|\nu) = x^{\alpha}s_{\nu}[x_{\ell(\mu)+1}+\ldots + x_{n}].$$ Putting this together we find that for all n sufficiently large
    \begin{align*}
        &\sum_{\substack{(\alpha|\nu) \\ \ell(\alpha) \leq \ell(\mu)}} M_{(\alpha|\nu)}^{(\mu|\lambda)} x^{\alpha}s_{\nu}[x_{\ell(\mu)+1}+\ldots + x_{n}]\\
        &= \Char\mathcal{V}^{(n)}(\mu|\lambda)\\
        &= \Char  \bigoplus_{\substack{(\alpha|\nu) \\ \ell(\alpha) \leq \ell(\mu)}} \chi^{(n)}(\alpha|\nu)^{\oplus d^{(n)}_{(\alpha|\nu)}} \\
        &= \sum_{\substack{(\alpha|\nu) \\ \ell(\alpha) \leq \ell(\mu)}} \Char\chi^{(n)}(\alpha|\nu)^{\oplus d^{(n)}_{(\alpha|\nu)}}\\
        &= \sum_{\substack{(\alpha|\nu) \\ \ell(\alpha) \leq \ell(\mu)}} d^{(n)}_{(\alpha|\nu)} \Char\chi^{(n)}(\alpha|\nu)\\
        &= \sum_{\substack{(\alpha|\nu) \\ \ell(\alpha) \leq \ell(\mu)}} \Dim \Hom_{\levi_{n}(\ell(\mu))}\left( \chi^{(n)}(\alpha|\nu),\mathcal{V}^{(n)}(\mu|\lambda) \right) x^{\alpha}s_{\nu}[x_{\ell(\mu)+1}+\ldots + x_{n}].\\
    \end{align*}
    Lastly, as the terms $x^{\alpha}s_{\nu}[x_{\ell(\mu)+1}+\ldots + x_{n}]$ for $\ell(\alpha) \leq \ell(\mu)$ are linearly independent we may compare coefficients to obtain the result.
\end{proof}

As a consequence of the above theorem we obtain a second proof of Theorem \ref{Positivity for almost symmetric Kostka coefficients}.

\begin{cor}\label{formula for monomial expansion in terms of schur expansion}
Let $(\mu|\lambda),(\alpha|\gamma)\in \Sigma.$  For all sufficiently large n
   $$|L_{(\alpha|\nu)}(\mu|\lambda)|  = \sum_{\gamma \in \Par} |\SSYT(\gamma,\nu)| \times \Dim \Hom_{\levi_{n}(\ell(\mu))}\left( \chi^{(n)}(\alpha|\gamma),\mathcal{V}^{(n)}(\mu|\lambda) \right) \in \mathbb{Z}_{\geq 0}.$$ 
\end{cor}
\begin{proof}
First, we expand the Schur functions $s_{\gamma}[\mathfrak{X}_{\ell(\mu)}] $ into the monomial symmetric function basis:
    \begin{align*}
        s_{(\mu|\lambda)}&= \sum_{\substack{(\alpha|\gamma) \\ \ell(\alpha) \leq \ell(\mu)}} M_{(\alpha|\gamma)}^{(\mu|\lambda)} x^{\alpha}s_{\gamma}[\mathfrak{X}_{\ell(\mu)}] \\
        &= \sum_{\substack{(\alpha|\gamma) \\ \ell(\alpha) \leq \ell(\mu)}} M_{(\alpha|\gamma)}^{(\mu|\lambda)} x^{\alpha}\sum_{\nu \in \Par} K_{\gamma, \nu} m_{\nu}[\mathfrak{X}_{\ell(\mu)}] \\
        &= \sum_{\substack{(\alpha|\nu) \\ \ell(\alpha) \leq \ell(\mu)}} \left( \sum_{\gamma \in \Par} K_{\gamma, \nu} M_{(\alpha|\gamma)}^{(\mu|\lambda)}  \right) x^{\alpha} m_{\nu}[\mathfrak{X}_{\ell(\mu)}].\\
    \end{align*}
    From here we find 
    $$K_{(\alpha|\nu)}^{(\mu|\lambda)} = \sum_{\gamma \in \Par} K_{\gamma, \nu} M_{(\alpha|\gamma)}^{(\mu|\lambda)}.$$ Lastly, by combining the formula $K_{\gamma, \nu} = |\SSYT(\gamma,\nu)|$, the expression for $M_{(\alpha|\gamma)}^{(\mu|\lambda)}$ in Theorem \ref{rep theory for schur expansion}, and the equation $K_{(\alpha|\nu)}^{(\mu|\lambda)} = |L_{(\alpha|\nu)}(\mu|\lambda)|$ from the proof of Theorem \ref{Positivity for almost symmetric Kostka coefficients} we conclude the desired result.
\end{proof}

\begin{remark}
    The \textbf{\textit{inverse Kostka coefficients}} $K^{(-1)}_{\gamma, \lambda}$ are given by 
    $$m_{\gamma} = \sum_{\lambda} K^{(-1)}_{\gamma, \lambda} s_{\lambda}.$$ Notice that 
    $$\delta_{\gamma,\lambda} = \sum_{\mu}K^{(-1)}_{\gamma, \mu} K_{\mu,\lambda}.$$ The inverse Kostka coefficients are known from the work of E$\tilde{g}$eciol$\tilde{g}$u-Remmel \cite{lu1990ACI} to have an explicit combinatorial formula involving \textbf{\textit{signed rim hook tabloids}} which we will not detail here. In the same way we obtained Corollary \ref{formula for monomial expansion in terms of schur expansion} we may instead expand each $m_{\lambda}$ into the Schur basis to obtain for all sufficiently large $n$
    $$\Dim \Hom_{\levi_{n}(\ell(\mu))}\left( \chi^{(n)}(\alpha|\nu),\mathcal{V}^{(n)}(\mu|\lambda) \right) = \sum_{\gamma \in \Par} K^{(-1)}_{\gamma, \nu} \times |L_{(\alpha|\gamma)}(\mu|\lambda)|.$$ Using the combinatorial formula for the $K^{(-1)}_{\gamma, \lambda}$ we see that this gives a purely combinatorial formula. However, this is \textbf{not} a non-negative combinatorial formula as the inverse Kostka coefficients are often negative. It would be interesting to find a non-negative combinatorial formula for the $M_{(\alpha|\gamma)}^{(\mu|\lambda)}.$

    Lastly we remark that by carefully taking direct limits of groups and their corresponding modules in the right way it is possible to simplify the expression in Theorem \ref{rep theory for schur expansion}: 
    $$M_{(\alpha|\gamma)}^{(\mu|\lambda)} = \Dim \Hom_{\levi_{\infty}(\ell(\mu))}\left( \chi^{(\infty)}(\alpha|\nu),\mathcal{V}^{(\infty)}(\mu|\lambda) \right).$$ 
\end{remark}

\printbibliography

@article{Dem_1974,
author = {Michel Demazure},
title = {{Désingularisation des variétés de Schubert généralisées}},
volume = {7},
journal = {Annales scientifiques de l'École Normale Supérieure},
number = {1},
pages = {53 -- 88},
year = {1974}
}

@article{mason_2009,
author = {Sarah Mason},
title = {{An Explicit Construction of Type A Demazure Atoms}},
URL = {https://arxiv.org/abs/0707.4267},
year = {2009}
}

@article{Ion_2003,
author = {Bogdan Ion},
title = {{Nonsymmetric Macdonald polynomials and Demazure characters}},
volume = {116},
journal = {Duke Mathematical Journal},
number = {2},
publisher = {Duke University Press},
pages = {299 -- 318},
year = {2003},
doi = {10.1215/S0012-7094-03-11624-5}
}

@article{Kirillov_2016,
   title={Notes on Schubert, Grothendieck and Key Polynomials},
   ISSN={1815-0659},
   url={http://dx.doi.org/10.3842/SIGMA.2016.034},
   DOI={10.3842/sigma.2016.034},
   journal={Symmetry, Integrability and Geometry: Methods and Applications},
   publisher={SIGMA (Symmetry, Integrability and Geometry: Methods and Application)},
   author={Kirillov, Anatol N.},
   year={2016},
   month=mar }

@article{goodberry2023geometric,
   title={A Geometric Realization of Partially-Symmetric Macdonald Polynomials},
   url={https://arxiv.org/abs/2312.11657},
   author={Ben Goodberry and Daniel Orr},
   year={2023}}

@article {CM_2015,
    AUTHOR = {Carlsson, Erik and Mellit, Anton},
     TITLE = {A proof of the shuffle conjecture},
   JOURNAL = {J. Amer. Math. Soc.},
  FJOURNAL = {Journal of the American Mathematical Society},
    VOLUME = {31},
      YEAR = {2018},
    NUMBER = {3},
     PAGES = {661--697},
      ISSN = {0894-0347},
   MRCLASS = {05E10 (05E05 33D52)},
  MRNUMBER = {3787405},
MRREVIEWER = {Tanja Stojadinovi\'{c}},
       DOI = {10.1090/jams/893},
       URL = {https://doi.org/10.1090/jams/893},
}

@article {GCM_2017,
    AUTHOR = {Carlsson, Erik and Gorsky, Eugene and Mellit, Anton},
     TITLE = {The {$\Bbb{A}_{q,t}$} algebra and parabolic flag {H}ilbert
              schemes},
   JOURNAL = {Math. Ann.},
  FJOURNAL = {Mathematische Annalen},
    VOLUME = {376},
      YEAR = {2020},
    NUMBER = {3-4},
     PAGES = {1303--1336},
      ISSN = {0025-5831},
   MRCLASS = {16W50 (05E16 14C05 16E35 19G38 33C52 57K18)},
  MRNUMBER = {4081116},
MRREVIEWER = {Primo\v{z} Moravec},
       DOI = {10.1007/s00208-019-01898-1},
       URL = {https://doi.org/10.1007/s00208-019-01898-1},
}

@article{Ion_2022, title={THE STABLE LIMIT DAHA AND THE DOUBLE DYCK PATH ALGEBRA}, DOI={10.1017/S1474748022000445}, journal={Journal of the Institute of Mathematics of Jussieu}, publisher={Cambridge University Press}, author={Ion, Bogdan and Wu, Dongyu}, year={2022}, pages={1–46}}

@article {C_2001,
    AUTHOR = {Cherednik, Ivan},
     TITLE = {Double affine {H}ecke algebras and difference {F}ourier
              transforms},
   JOURNAL = {Invent. Math.},
  FJOURNAL = {Inventiones Mathematicae},
    VOLUME = {152},
      YEAR = {2003},
    NUMBER = {2},
     PAGES = {213--303},
      ISSN = {0020-9910},
   MRCLASS = {20C08 (05E05 17B37)},
  MRNUMBER = {1974888},
       DOI = {10.1007/s00222-002-0240-0},
       URL = {https://doi.org/10.1007/s00222-002-0240-0},
}

@article {haglund2007combinatorial,
    AUTHOR = {Haglund, J. and Haiman, M. and Loehr, N.},
     TITLE = {A combinatorial formula for nonsymmetric {M}acdonald
              polynomials},
   JOURNAL = {Amer. J. Math.},
  FJOURNAL = {American Journal of Mathematics},
    VOLUME = {130},
      YEAR = {2008},
    NUMBER = {2},
     PAGES = {359--383},
      ISSN = {0002-9327},
   MRCLASS = {05E05},
  MRNUMBER = {2405160},
MRREVIEWER = {Pavlo Pylyavskyy},
       DOI = {10.1353/ajm.2008.0015},
       URL = {https://doi.org/10.1353/ajm.2008.0015},
}

@misc{lapointe2022msymmetric,
      title={$m$-Symmetric functions, non-symmetric Macdonald polynomials and positivity conjectures}, 
      author={Luc Lapointe},
      year={2022},
      eprint={2206.05177},
      archivePrefix={arXiv},
      primaryClass={math.CO}
}

@book {Macdonald,
    AUTHOR = {Macdonald, I. G.},
     TITLE = {Symmetric functions and {H}all polynomials},
    SERIES = {Oxford Classic Texts in the Physical Sciences},
   EDITION = {Second},
      NOTE = {With contribution by A. V. Zelevinsky and a foreword by
              Richard Stanley,
              Reprint of the 2008 paperback edition [ MR1354144]},
 PUBLISHER = {The Clarendon Press, Oxford University Press, New York},
      YEAR = {2015},
     PAGES = {xii+475},
      ISBN = {978-0-19-873912-8},
   MRCLASS = {05E05 (01A75 05-02 20C30 20C33 20K01 33C80 33D80)},
  MRNUMBER = {3443860},
}

@phdthesis{Goodberry,
    author={Goodberry, Ben},
    title={Partially-Symmetric Macdonald Polynomials},
    school={Virginia Polytechnic Institute and State University},
    year={2022}
    }

@article{lu1990ACI,
  title={A combinatorial interpretation of the inverse kostka matrix},
  author={Egecioglu, O. and Jeffrey B. Remmel},
  journal={Linear \& Multilinear Algebra},
  year={1990},
  volume={26},
  pages={59-84},
  url={https://api.semanticscholar.org/CorpusID:122351577}
}

@article{Andersen1985,
author = {Andersen, H.H.},
journal = {Inventiones mathematicae},
pages = {611-618},
title = {Schubert varieties and Demazure's character formula.},
url = {http://eudml.org/doc/143216},
volume = {79},
year = {1985},
}

@article{MBWArxiv,
      title={Stable-Limit Non-symmetric Macdonald Functions}, 
      author={Bechtloff Weising, Milo},
      year={2023},
      eprint={2307.05864},
      archivePrefix={arXiv},
      primaryClass={math.RT},
      URL = {https://arxiv.org/abs/2307.05864}
}

\end{document}